\documentclass[11pt]{article}
\usepackage[english]{babel}
\usepackage[utf8]{inputenc}
\usepackage{amsmath}
\usepackage{amsthm,amsfonts,amssymb,mathrsfs,dsfont}

\usepackage[mono=false]{libertine}
\useosf

\usepackage{wasysym}

\usepackage[dvipsnames]{xcolor}
\usepackage[a4paper,vmargin={2.6cm,2.7cm},hmargin={2.6cm,2.6cm}]{geometry}
\usepackage[font=sf, labelfont={sf,bf}, margin=1cm]{caption}

\usepackage{lscape}
\usepackage{pdflscape}
\usepackage{graphicx}

\usepackage[font=footnotesize,labelfont=bf]{caption}
\usepackage{subfig} 
\usepackage{epsfig}
\usepackage{latexsym}
\usepackage{stmaryrd}
\usepackage{ae,aecompl}

\usepackage[english]{babel}

 \usepackage[colorlinks=true]{hyperref}
\usepackage{pstricks}
\usepackage{enumerate}
\usepackage[normalem]{ulem}

\theoremstyle{plain}

\usepackage[font=footnotesize,labelfont=bf]{caption}

\theoremstyle{plain}
\newtheorem{theorem}{Theorem}[section]
\newtheorem{corollary}[theorem]{Corollary}

\newtheorem{lemma}[theorem]{Lemma}
\newtheorem{proposition}[theorem]{Proposition}

\theoremstyle{definition}

\newtheorem{question}[theorem]{Question}

\newtheorem{remark}[theorem]{Remark}

\makeatletter
\newcommand*\bigcdot{\mathpalette\bigcdot@{.5}}
\newcommand*\bigcdot@[2]{\mathbin{\vcenter{\hbox{\scalebox{#2}{$\m@th#1\bullet$}}}}}
\makeatother

\renewcommand{\leq}{\leqslant}
\renewcommand{\geq}{\geqslant}

\newcommand{\N}{\mathbb{N}}

\newcommand{\R}{\mathbb{R}}

\newcommand{\ind}[1]{\mathbf{1}_{\left\{#1\right\}}}

\newcommand{\floor}[1]{{\left\lfloor #1 \right\rfloor}}

\DeclareMathOperator{\E}{\mathbf{E}}
\renewcommand{\P}{\mathbf{P}}

\renewcommand{\bar}[1]{\overline{#1}}
\renewcommand{\tilde}[1]{\widetilde{#1}}
\renewcommand{\epsilon}{\varepsilon}
\renewcommand{\phi}{\varphi}
\newcommand{\T}{\mathcal{T}}
\newcommand{\Tfull}{\mathcal{T}_{\hbox{\tiny{full}}}}

\newcommand{\defeq}{:=}

\newcommand{\K}{\mathbf{K}}

\usepackage{xcolor}

\newcommand{\fils}[2]{{#1}^{#2}}

\newcommand{\pere}[1]{{#1}^{0}}

\title{A phase transition for the biased tree-builder random walk}

\author{Arthur Blanc-Renaudie\thanks{CNRS and LMRS, Universit\'e de Rouen. Email: \url{ablancrenaudiepro@gmail.com}}, Camille Cazaux\thanks{LPSM, Sorbonne Universit\'e. Email: \url{ccazaux@lpsm.paris}}, \\  Guillaume Conchon-Kerjan\thanks{Department of Mathematics, King's College London. Email: \url{guillaume.conchon-kerjan@kcl.ac.uk}}, Tanguy Lions\thanks{ENS Lyon. Email: \url{tanguy.lions@ens-lyon.fr}}, Arvind Singh\thanks{CNRS and LMO, Université Paris-Saclay. Email: \url{arvind.singh@universite-paris-saclay.fr} Partially supported by \texttt{ANR 19-CE40-0025} ``ProGraM''}}

\date{\today}

\begin{document}

\maketitle

\begin{abstract}
We consider a recent model of random walk that recursively grows the network on which it evolves, namely the Tree Builder Random Walk (TBRW). We introduce a bias $\rho \in (0,\infty)$ towards the root, and exhibit a phase transition for transience/recurrence at a critical threshold $\rho_c =1+2\overline{\nu}$, where $\overline{\nu}$ is the (possibly infinite) expected number of new leaves attached to the walker's position at each step. This generalizes previously known results, which focused on the unbiased case $\rho=1$. 
\\
The proofs rely on a recursive analysis of the local times of the walk at each vertex of the tree, after a given number of returns to the root. 
\\
We moreover characterize the strength of the transience (law of large numbers and central limit theorem with positive speed) via standard arguments, establish recurrence at $\rho_c$, and show a condensation phenomenon in the non-critical recurrent case. 
\end{abstract}

\smallskip


\noindent\textbf{AMS Classification 2020:} Primary 60K35; 60K37 Secondary 05C80

\smallskip

\section{Introduction}

The Tree Builder Random Walk (TBRW), introduced in full generality in~\cite{TBRWZuaznabar}, is a prime instance of a random walk that modifies the local structure of the network along its trajectory, see e.g.~\cite{AMIRal} for a broad formulation and~\cite{NRRW} for a partly intersecting model. In simple words, the TBRW is a random walk on a tree such that at time-step $n$, a random number $\xi_n$ of leaves are attached to the current position of the walker. When $(\xi_n)_{n\geq 1}$ is i.i.d., the TBRW is known to be transient and even ballistic~\cite{TBRWtransience,TBRWZuaznabar}, with a positive speed~\cite{TBRWtransience,ribeiroTBRW} and a renewal structure~\cite{ribeiroTBRW}.

Hence, it is natural to seek to increase the chances of returning to the root. One way to do it is to decrease the probability to grow a new vertex at the $n$-th time step, i.e. $\xi_n\sim \text{Ber}(p_n)$, which makes the TBRW recurrent if $p_n=n^{-\gamma}$ with $\gamma>1/2$~\cite{TBRWEnglander}. When $\gamma \in (2/3,1]$, the tree on which the random walk evolves grows like a Barab\'asi-Albert preferential attachment model~\cite{TBRWEnglander2}. Note that the case $\gamma \leq 1/2$ is still open, with the walk being conjectured to be transient for $\gamma <1/2$. 

Another intuitive option is to add a bias towards the root for the walker's trajectory, and this is the purpose of the present paper. Due to the local structure changing around the walker, this model is noticeably different from usual biased random walks, e.g. on Galton-Watson trees (see~\cite{LPP} and the rich literature that followed).

We consider the following model which constructs a growing sequence $(\T_n,\;n\geq 0)$ of random trees together with a random walk $(S_n,\;n\geq 0)$ moving on those trees. 

\medskip

Fix $\rho > 0$ and a probability distribution $\nu$ on $\N = \{0,1,\ldots\}$ which are the two parameters of the model. In the following, we always assume that $\nu(0) < 1$ (otherwise the model is trivial). At time~$0$, the tree $\T_0$ consists of a single vertex $o$ (the root) and the position of the walk is $S_0 =  o$.  Then, by induction, given $\T_n$ and $S_n$: 
\begin{itemize}
\item First, independently of everything else, we grow $\T_{n+1}$ from $\T_n$ by adding a random number $\xi_n$ of leaves drawn according to $\nu$. All these leaves are attached to the vertex $S_n$. In particular, with probability $\nu(0)$, we have $\T_{n+1} = \T_n$. 
\item Second, we define $S_{n+1}$ as a neighbor of $S_{n}$ in $\T_{n+1}$ chosen with bias $\rho$ toward the root. More precisely, if $S_n  = v \in  \T_n$  and if we denote by $\pere{v}$ the father  of $v$ and by $\fils{v}{1},\ldots, \fils{v}{k}$ the children of $v$ in $\T_{n+1}$, then 
$$
\P(S_{n+1} = \fils{v}{i} \;|\; \T_{n+1}, S_n) = 
\begin{cases}
\frac{\rho}{k+\rho}&\hbox{if $i=0$,}\\
\frac{1}{k+\rho}&\hbox{if $1\leq i\leq k$.}
\end{cases}
$$
By convention, the father of the root $o$ is the root itself so that we do not have to treat this case separately. This means that when the walker is at the root, it has positive probability of staying at the root at the next time step - this amounts to adding a loop around $o$ that the walker can take with the corresponding probability. 
\end{itemize}

When $\rho = 1$, $S_n$ performs a simple random walk on the current tree, and this corresponds to the original model of the TBRW with $s=1$ and $(\xi_n)$ i.i.d., where $s$ is the number of consecutive steps taken by the walk between two additions of new leaves. In this setting, it is known that for any distribution $\nu$, the walk is transient and ballistic~\cite{TBRWZuaznabar, ribeiroTBRW}.

In this paper, we demonstrate how varying the bias $\rho$ towards the root affects the geometry of the tree $\T_n$ and the behavior of the walk $S_n$. Our main result is the existence of a non-trivial phase transition for this model depending on $(\rho, \nu)$. 

\begin{theorem}\label{maintheo1}
Let $\bar{\nu}$ denote the (possibly infinite) expectation of $\nu$:
$$
\bar{\nu} = \sum_{k=0}^\infty k \nu(k)  \in (0,\infty].
$$
For any $\rho > 0$, the following holds.
\begin{itemize}
\item If $\rho < 1+ 2 \bar{\nu}$, then walk $S_{n}$ is transient, \emph{i.e.} each vertex created is visited only finitely many times by the walk a.s. In particular, the walk returns only finitely many times to the root a.s.
\item If $\rho \geq 1 + 2\bar{\nu}$, then the walk $S_{n}$ is recurrent, \emph{i.e.} every vertex is visited infinitely often a.s. by the walk. Furthermore, 
\begin{itemize}
    \item If $p = 1 + 2\bar{\nu}$, then the walk is ``null recurrent'' in the sense that all return times to the root have infinite expectation. 
    \item If $p > 1 + 2\bar{\nu}$, then the walk is ``positive recurrent'' in the sense that  all return times to the root have finite expectation.
\end{itemize}
\end{itemize}
\end{theorem}

\begin{remark}
\begin{itemize}
    \item When $\rho =1$, Theorem \ref{maintheo1} recovers the fact that the walk is always transient, irrespectively of the distribution $\nu$ (see \cite[Theorem 1.3]{TBRWZuaznabar}, which also holds for $s$ odd and $(\xi_n)$ uniformly elliptic - interestingly, \cite[Theorem 1.2]{TBRWZuaznabar} proves that for $s$ even and under mild assumptions on $(\xi_n)$, the walker will be recurrent). 
    \item Conversely, for any distribution $\nu$ with finite expectation, the walk becomes recurrent whenever $\rho$ is chosen large enough. This may seem surprising at first glance because, if we assume the walk to be recurrent, then the degree of every given vertex of the tree must ultimately increase to infinity, inducing local bias against the root which, intuitively, should in turn make the walk transient. A rough explanation for this apparent contradiction comes from the fact that, in the recurrent setting, even though all vertices have their degrees increasing to infinity, because the walk creates new vertices all the time, most of the vertices present at any given time have a current degree which is smaller than $\rho$, hence they induce a local bias toward the root which counterbalances the opposite bias created by the (few) vertices of large degree. 
    \item When the walk is recurrent, every vertex created is visited infinitely often a.s. By a straightforward application of the Borel-Cantelli Lemma, this implies that every vertex has as a degree which increase to infinity so the sequence of tree $\T_n$ converge locally to an infinite tree where every vertex has a (countably) infinite number of children, and yet, all of its vertices are visited infinitely often. 
    \item We point out that the walk $S_n$ is not a Markov process (but the pair $(\T_n, S_n)$ is Markov) so the meaning of ``positive recurrence" when $\rho > 1 + 2\bar{\nu}$ differs from that for usual Markov chains. For instance, the excursions of the tree-builder walk $S_n$ between two returns to the root are not i.i.d. as in the Markov case and the expectation of the $k$-th return time to the root grows faster than linearly (in fact exponentially in $k$, \emph{c.f.} Lemma \ref{lemma:EN} and Lemma \ref{lemma:exptau}) and
    $$
    \frac{1}{n}\sum_{k=0}^n \ind{S_n = o} \underset{n\to\infty}{\longrightarrow} 0 \quad{a.s.}
    $$
    In particular, the ergodic theorem for Markov chain does not apply here and the distribution of $S_n$ does not converge to a non-degenerate distribution.
\end{itemize}
\end{remark}

By definition, the number of vertices added to the tree at each step is given by a sequence of i.i.d. r.v. distributed as $\nu$. Therefore, by the strong law of large numbers (assuming $\nu$ has finite expectation), the number of vertices $\sharp\T_n$ of $\T_n$ satisfies:
$$
\sharp\T_n \underset{n\to\infty}{\sim} \bar{\nu} n \quad \hbox{a.s.}
$$
The asymptotic geometry of the tree $\T_n$ differs depending on the transience/recurrence regime considered. In the transient case,  we prove that the walk is ballistic, hence the tree is asymptotically a line onto which a grafted finite random trees. We use the notation $|v|$ for the height of a vertex (\emph{i.e.} its distance from the root) and define the maximum height of the tree $\T_n$ by
$$
|\T_n| \defeq \max_{v\in\T_n}|v|.
$$
\begin{theorem}\label{theo:LLN_TCL} Suppose that $\rho < 1 + 2\bar{\nu}$. There exist positive constants $v = v(\rho, \nu)$ and $\sigma  = \sigma(\rho, \nu)$ such that
\begin{equation}\label{eq:LGNandTCL}
\frac{|S_n|}{n}  \overset{\hbox{\tiny{a.s.}}}{\underset{n\to\infty}{\longrightarrow}} v  \qquad\hbox{and}\qquad \frac{|S_n| - v}{\sqrt{n}}  \overset{\hbox{\tiny{law}}}{\underset{n\to\infty}{\longrightarrow}} \mathcal{N}(0,\sigma^2).
\end{equation}
The same results also hold with $|\T_n|$ in place of $|S_n|$ in \eqref{eq:LGNandTCL}.
\end{theorem}
This generalizes~\cite[Theorem 1.1]{TBRWtransience} and~\cite[Theorems 3 and 4(a)]{ribeiroTBRW}.

On the other hand, in the positive recurrent case, we observe a condensation of the tree whose  height now grows logarithmically in $n$. 

\begin{figure}
\begin{center}
\subfloat[$\rho = 2.95$]{\includegraphics[height=3.4cm]{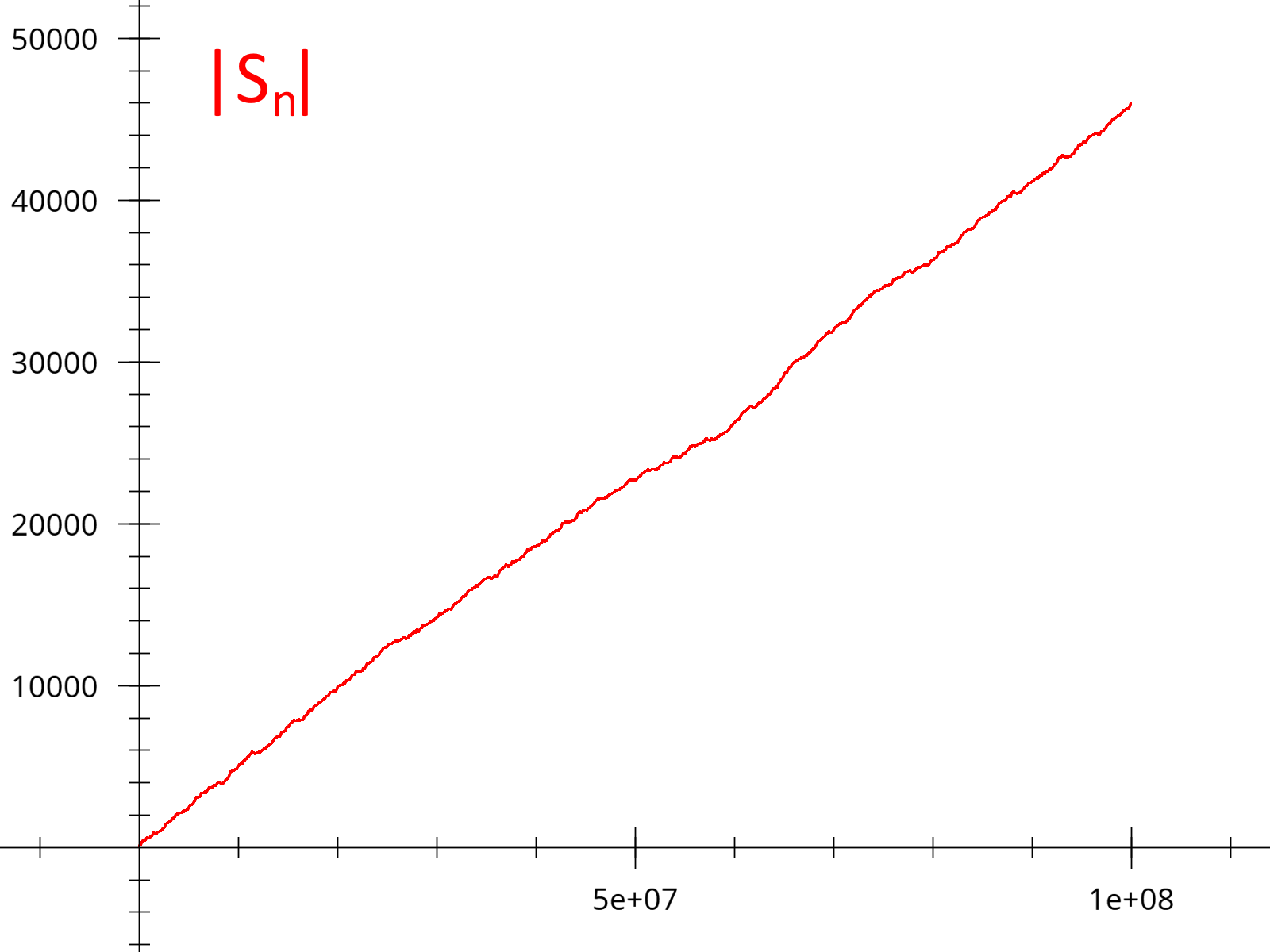}}
\hspace{0.4cm}
\subfloat[\label{fig:sim_b}$\rho = 3$]{\includegraphics[height=3.4cm]{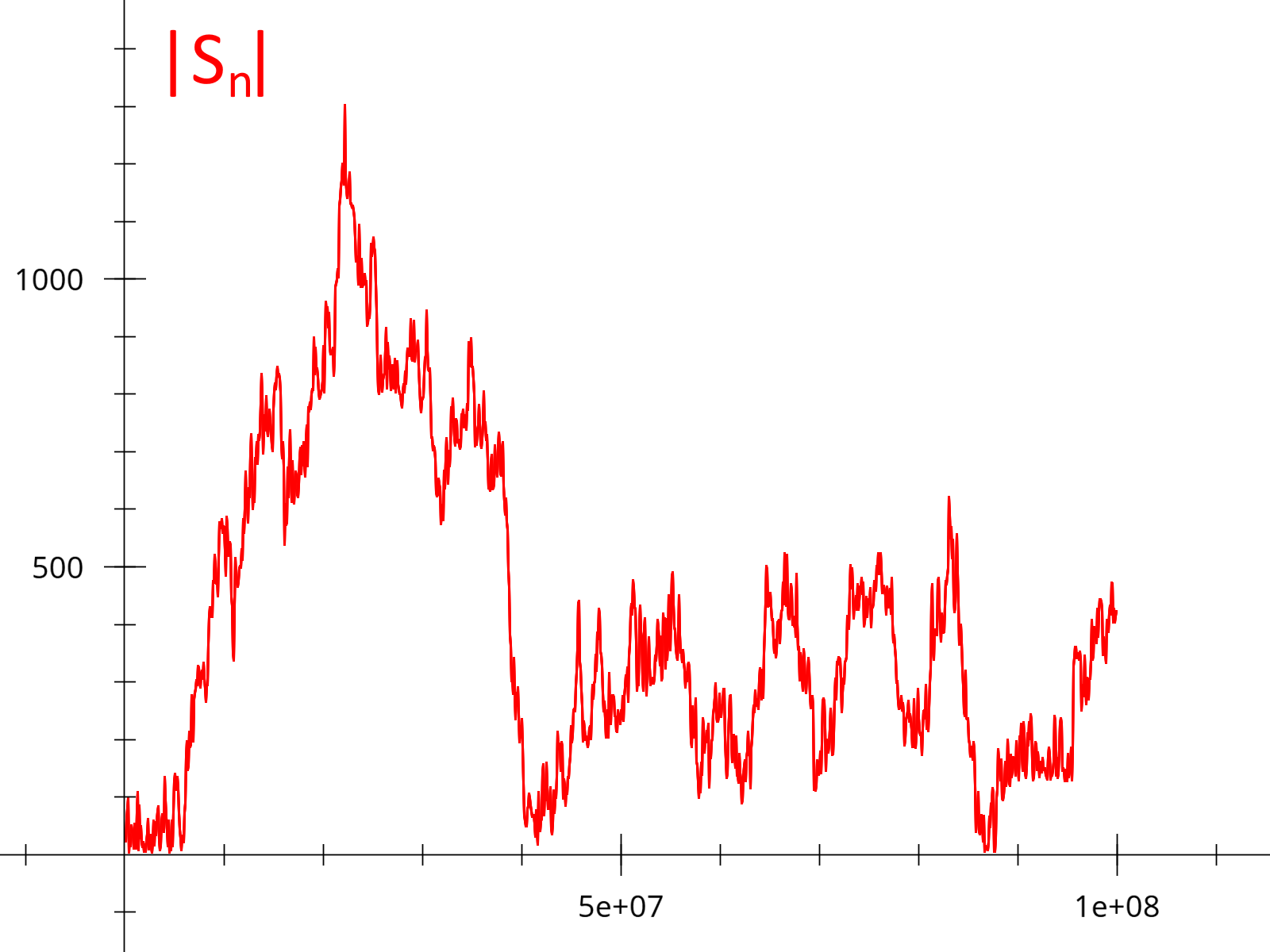}}
\hspace{0.4cm}
\subfloat[$\rho = 3.05$]{\includegraphics[height=3.4cm]{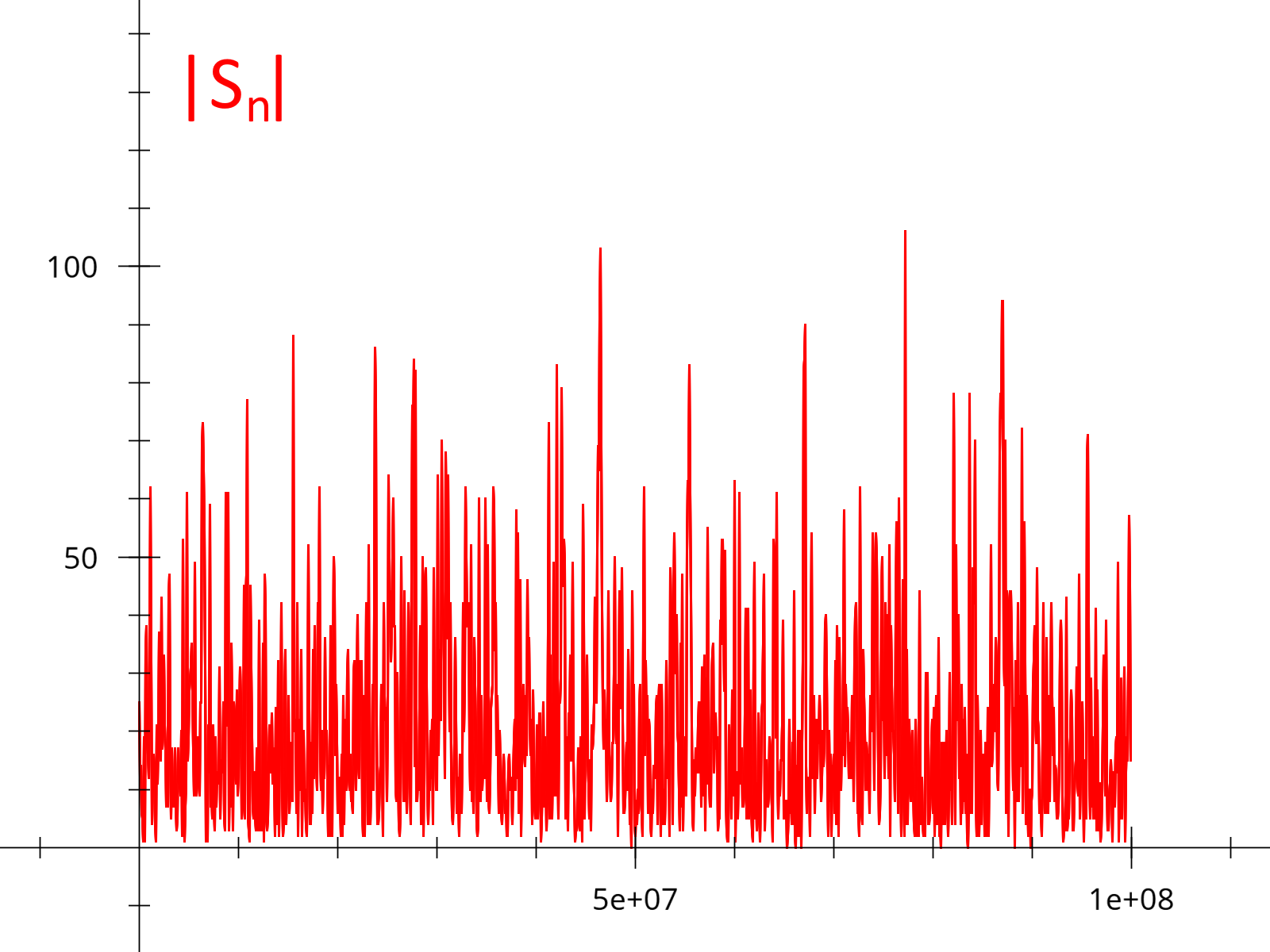}}
\end{center}
\caption{\label{fig:sim}Simulation of the first $10^8$ steps of the height $|S_n|$ of the tree-builder random walk when $\nu(k) = \mathbf{1}_{k=1}$ (\emph{i.e.} a new leaf is created at each step) in the transient (a) , null recurrent (b) and positive recurrent (c) regime. Here, the phase transition occurs at the critical parameter $\rho_c = 1 + 2\bar{\nu} = 3$.}
\end{figure}

\begin{theorem}\label{theo:rec_log}
Suppose that $\rho >  1 + 2\bar{\nu}$. There exists positive constants $c_1 =  c_1(\rho, \nu)$ and $c_2 =  c_2(\rho, \nu)$ such that
$$
c_1 \log n \leq |\T_n| \leq c_2 \log n\qquad \hbox{for all $n$ large enough a.s.}
$$
\end{theorem}

In this paper, we leave out the study of the asymptotical behavior of the walk and the tree in the critical case. 
\begin{question} 
What is the typical height $|S_n|$ and diameter $|\T_n|$ in the critical case $\rho =  1 + 2\bar{\nu}$ ? Numerical simulation suggest that these quantities should be of order $\sqrt{n}$ as $n$ tend to infinity, \emph{c.f.}  Figure~\ref{fig:sim_b}. 
\end{question}

The rest of the paper is organized as follows. In section \ref{sec:sec2}, we first set some notations and then prove a monotonicity property (Proposition \ref{prop:coupling_nurho}) for the recurrence/transience of the TBRW with respect to its parameter $(\rho,\nu)$ via a coupling argument. This allows us, later on, to prove the recurrence/transience criterion under the assumption that $\nu$ has finite expectation and subsequently extend it, by comparison, to the case where $\bar\nu = \infty$.

Next, we introduce two new processes: an urn process (section \ref{subsec:urnprocess}) and a branching Markov chain (section \ref{sec:BMC_Z}). These processes  are closely related to the local times of the tree-builder walk (Proposition \ref{prop:constructurn} and Proposition \ref{prop:loctimerec}) and reduce the study of the recurrence/transience of the TBRW to the survival of a particular ``infinite-type'' branching process. Section \ref{sec:sec3} is devoted to proving Theorem \ref{maintheo1} which is done by computing a (strikingly explicit) quasi-stationary distribution for the multi-type branching process, \emph{c.f.} Corollary \ref{cor:vecpropre}. Then, in section \ref{sec:sec4}, we recall the classical method of cut/regeneration times and apply it to prove Theorem \ref{theo:LLN_TCL} and then carry on proving Theorem \ref{theo:rec_log}.

\section{An urn process and a branching Markov chain related to the TBRW.\label{sec:sec2}}

\subsection{Notations}
In the rest of the paper, $(S_n, \T_n)$ always denotes the TBRW and associated tree at step $n$. Since the sequence $(\T_n,\;n\geq 0)$ of random trees is non-decreasing (we do not remove vertices), it makes sense to consider the inductive limit 
$$
\T_\infty \defeq \bigcup_n \T_n.
$$
We point out that $\T_\infty$ is a (random) rooted tree but it needs not be locally finite: a vertex may have an infinite number of children. By the Borel-Cantelli Lemma, this happens i.f.f. $S_n$ visits this vertex infinitely many times since, at each visit and independently of everything happening before, it has probability $1 -\nu(0) > 0$ of creating at least one new child.
 
The trees $\T_n$ (and  also $\T_\infty$) have a natural planar embedding, i.e. an ordering of the children of any vertex, given by their order of creation. Let us set some notation. For a vertex $v \in \T_n$, we define $\sharp_n v $ to be the number of children of $v$ in $\T_n$, and we denote $\fils{v}{1},\fils{v}{2},\ldots$ the children of $v$ (ordered by time of creation) and by $\pere{v}$ the father of $v$. By convention, we set $\pere{o} = o$. Note that by compatibility of the trees $\T_n$, we do not need to add a subscript $n$ to the children (or the father) of $v$ as they are the same in all trees where they exist. Given a vertex $v$, we write $|v|$ for its height defined as the distance between this vertex and the root $o$. Once again, this definition does not depend on the tree considered (provided $v$ exists in it). 

Given a rooted tree $\T$ and $d\in \N$, we define $\T^{\leq d}$ the subtree composed of all vertices of height at most~$d$. Finally, we denote by $\Tfull$ the ``full'' planar tree where every vertex $v$ has an infinite number of children $\fils{v}{1},\fils{v}{2},\ldots$. Any planar tree $\T$ can therefore be seen as a subtree of~$\Tfull$.

For two probability distributions $\nu,\tilde{\nu}$ (on $\N$), we write $\nu \prec\tilde{\nu}$, meaning that $\tilde{\nu}$ dominates $\nu$ for the stochastic partial order, if $\tilde\nu([k,\infty))\geq \nu([k,\infty))$ for all $k\geq 0$.

\subsection{Coupling of walks for different values of the parameters\label{sec:coupling}}

One should expect the tree-builder walks to become ``more recurrent'' when $\rho$ increases or when the distribution $\nu$ of the number of children created at each step decreases (stochastically). In order to formalize this idea, we construct couplings between processes with different parameters. 

\begin{proposition}\label{prop:coupling_nurho}
Fix $\rho \geq \tilde{\rho} >0$ and two distributions $\nu \prec\tilde{\nu}$. If the $(\tilde{\rho}, \tilde{\nu})$-TBRW goes through the root loop infinitely often a.s., then the $(\rho, \nu)$-TBRW does the same a.s.  
\end{proposition}

\begin{proof}
We construct two TBRWs $(S_n,\T_n)$ and $(\tilde{S}_n,\tilde{\T}_n)$ with respective parameters $(\rho,\nu)$ and $(\tilde\rho,\tilde{\nu})$ on the same probability space, by a natural extension of the coupling described in~\cite[Lemma 1]{ribeiroTBRW} for $\rho=\tilde{\rho}$ and $\nu =\text{Ber}(1-\tilde{\nu}(0))$. 

Roughly speaking, we decompose the trajectory of $\tilde{S}$ where it evolves together with $S$, and excursions in subtrees above a given position of $S$. This can happen either when $S$ moves towards the root while $\tilde{S}$ moves away, coupling their jumps while using that $\rho>\tilde{\rho}$, or when $\tilde
{S}$ moves to a child of $\tilde\T\setminus \T$, using that $\nu\prec \tilde\nu$. For any such excursion, either it never ends, in which case $\tilde{S}$ is transient, or it comes back to the position where we froze $S$. Then the two walks resume their joint evolution, until a next time when $\tilde{S}$ moves away from the root and $S$ towards it, and so on. We note that while $S$ is frozen, the excursions of $\tilde{S}$ will create new vertices on the tree $\tilde{\T}_n$ (compared to $\T_n$) but these additional vertices only make the walk $\tilde{S}$ more transient compared to $S$ so that they are not problematic.

We consider two i.i.d. sequences $(\xi_{k,v}, v\in\Tfull, k\in\N^*)$ and $(\tilde{\xi}_{k,v}, v\in\Tfull, k\in\N^*)$ such that $\xi_{k,v}$ (resp. $\tilde{\xi}_{k,v}$) has distribution $\nu$ (resp. $\tilde{\nu}$) and such that 
\begin{equation}\label{eq:xixitilde}
    \xi_{k,v} \leq \tilde{\xi}_{k,v}\text{ a.s.}, 
\end{equation}
which is possible since $\nu\prec \tilde{\nu}$. We will impose that at the $k$-th visit of $v$ by $S$ (resp.~$\tilde{S}$), we add $\xi_{k,v}$ (resp. $ \tilde{\xi}_{k,v}$) neighbours to $v$ in the corresponding TBRW trees. 

Construct $\tilde{S}$ according to the $(\tilde\rho,\tilde\nu)$-TBRW dynamics. We now construct $S$. 
Initially, we start with $S_0=\tilde{S}_0=o$ and $\T_1$ is $o$ with $\xi_{1,o}$ leaves attached to it. For $n\geq 0$, let $t_{n}$ be the first time $t> t_{n-1}$ such that $\tilde S_t=S_n=:v$, $\T_{n+1}\subseteq \tilde\T_{t+1}$, and $\tilde{S}_{t+1}$ is either $\pere{v}$, or a child of $v$ in $\T_{n+1}$, with the possibility that $t_n=\infty$. 

If $t_n<\infty$, there are two cases. If $\tilde{S}_{t+1}=\pere{v}$, we let $S_{n+1}=\pere{v}$. And if $\tilde{S}_{t+1}=:v'$ is a child of $v$ in $\T_{n+1}$, we let $x:=\xi_{1,v}+\ldots +\xi_{k,v}$, $S_n$ being the $k$-th visit of $v$ by $S$. With probability $b:=\frac{\rho}{\rho +x} -  \frac{\tilde\rho }{\tilde \rho +{x}}$ we let $S_{n+1}=\pere{v}$, and with probability $1-b$ we set $S_{n+1}=v'$. Note that $b\geq 0$ since $\rho \geq \tilde{\rho}$. 

By induction on $n$, the first $t>t_{n-1}$ such that  $\tilde S_t=S_n=:v$ and $\tilde{S}_{t+1}$ is either $\pere{v}$, or a child of $v$ in $\T_n$ also satisfies that $\T_{n+1}\subseteq \tilde\T_{t+1}$. Indeed, at any vertex of $\Tfull$ the local time $k$ of $(S_j)_{j\leq n}$ is no greater than that $\tilde{k}$ of $(\tilde{S}_j)_{j\leq t}$ since $S$ always waits for $\tilde{S}$ to catch it before moving again. By~\eqref{eq:xixitilde}, we have thus $\xi_{1,v}+\ldots +\xi_{k,v}\leq \tilde{\xi}_{1,v}+\ldots +\tilde\xi_{k,v}\leq \tilde{\xi}_{1,v}+\ldots +\tilde\xi_{\tilde k,v}$. Hence, $t$ is in fact $t_n$.

Moreover, in our construction, the first case happens with probability $\frac{\tilde\rho }{\tilde \rho +{x}}= \frac{\rho}{\rho +x} -b $ so that $S_{n+1}$ moves to $\pere
{v}$ with total probability $\frac{\rho}{\rho +x}$, and to any given child of $v$ in $\T_{n+1}$ with probability $x^{-1}(1-\rho/(\rho+x))=1/(\rho+x)$. Hence, $S$ performs a $(\rho,\nu)$-TBRW. Furthermore, one has easily by induction on $n$ that $S_n$ is always an ancestor of $\tilde{S}_t$ for $t_{n-1}+1\leq t\leq t_n$, and that if $\tilde{S}_t=S_n$ for some $t>t_{n-1}$, then $\tilde{S}$ only moves towards $o$ if so does $S$. Thus  
\begin{equation}\label{eq:visitsroot}
(S_j)_{j \leq n}\text{ visits $o$ at least as many times as }(\tilde{S}_j)_{j\leq t_n} .
\end{equation}
Taking $n=\min\{j, t_j=\infty\}-1$ if this minimum exists or $n\rightarrow \infty$ else yields the conclusion. 
\end{proof}

\subsection{The urn process associated with the local time at a vertex}\label{subsec:urnprocess}

We now introduce an auxiliary urn process, the \emph{TBRW urn}, which captures the local behavior of the TBRW around a fixed vertex and whose detailed analysis will play a critical role in the rest of the paper. Initially, we start with the urn containing $\rho$ balls\footnote{$\rho$ need not be an integer, so that the urn can contain a non-integer number of balls, but this does not matter from the mathematical point of view.} of color $0$. At each time step $n\geq 1$, we make the following operation.
\begin{enumerate}
\item[(a)] First, independently of everything else, we add a random number of new balls sampled according to the distribution $\nu$. Each added ball is given a new color, whose index is the smallest color that has not yet been attributed (so if the colors already present in the urn are $0,1,\ldots,k$ and we add $j$ new balls, then they are given colors $k+1, k+2,\ldots, k+j$). 
\item[(b)] Second, we draw a uniform random ball from the urn, we record its color $B_n$, and we put it back in the urn. 
\end{enumerate}
We construct in this way a random sequence $B_1,B_2\ldots$ of integers representing the colors of the balls drawn from the urn at each step.

This urn process and the TBRW are related in the following way: fix a vertex $v\in\Tfull$ and look at the times when the walk $S$ is at $v$ (if ever). The crucial observation is that when $S$ is at $v$, the displacement rule is local, meaning that it only depends on the current number of children of $v$ and on an independent additional randomness (to make one step given this local environment). But it does not depend on what happens to the walk when it is not at $v$. Moreover, conditionally on having $k$ children at $v$, the probability of jumping to $v^{i}$ for $i\in \{0,1,\ldots,k\}$ is equal to the probability of drawing the ball of color $i$ in the urn currently containing colors $0,1, \ldots, k$. Putting these arguments together, it is immediate to check that the TBRW can be constructed from a family of independent TBRW urns as follows:
\begin{proposition}[Construction of the TBRW from an i.i.d. family of urns]\label{prop:constructurn}
The TBRW $S$ can be constructed in the following way:
\begin{enumerate}
\item We initially put an urn $U_v$ for every $v\in \Tfull$, containing $\rho$ balls of color $0$. 
\item We start the walk at $S_0 = o$.
\item When $S$ is at a vertex $v$, we perform (a) and (b) above on the urn $U_v$ (independently of all past operations). $S$ 
then moves to the vertex $\fils{v}{B}$ where $B$ is the color of the ball drawn. 
\end{enumerate}
\end{proposition}
We insist on the fact that the urns are independent. In order to study the recurrence or transience of the walk, it is therefore crucial to understand the asymptotic behavior of the TBRW urn and, in particular, to check whether a given color is picked infinitely often  a.s.

\begin{lemma}\label{lemma:draw0}
Suppose that
\begin{equation}\label{eq:draw_k_times}
\P(\hbox{color $0$ is drawn a finite number of times from the TBRW urn}) > 0.
\end{equation}
Then,
$$
\P(\hbox{color $0$ is never drawn from the TBRW urn}) > 0.
$$
\end{lemma}
\begin{proof}
We can construct the urn process by first prescribing the number of balls added at each step according to an i.i.d. sequence $\xi=(\xi_i)_{i\geq 1}$ with distribution $\nu$ and then, conditionally on $\xi$, by performing the draws  $B_1,B_2,\ldots$ independently. We can also assume without loss of generality that $\xi_1 > 0$ since we can discard all the draws (which are necessarily of color $0$) until a non-zero color ball is added to the urn. Then the process after that time is the same as the initial urn process conditioned on $\xi_1 >0$. 

 Now, if \eqref{eq:draw_k_times} holds, there  must exist $k\geq 1$ and indices $i_1,i_2, \ldots i_k$ such that, with positive probability, color $0$ is drawn from the urn at times $i_1,\ldots,i_k$ and at no other time. For each draw $B_{i_j}$, we have  $\P(B_{i_j} = 0 \,|\, \xi) \P(B_{i_j} \neq 0 \,|\, \xi) > 0$  (since the urn always contains a non-zero color) This means that we can change the color $B_{i_j}$, for $j = 1,\ldots, k$, to non-zero values while still keeping an event of positive probability. On that new event, color $0$ is never drawn. 
\end{proof}

\begin{corollary}[Dichotomy for the recurrence/transience of the TBRW]\label{cor:rec01}
For any vertex $v\in \Tfull$, we have:
\begin{equation}\label{eq:vtovo}
\P(\hbox{$S$ visits $v$ i.o. but jumps from $v$ to $\pere{v}$ finitely many times} ) = 0.
\end{equation}
As a consequence, we have 
$$
\P\Big( \substack{\hbox{$S$ crosses the root loop} \\ \hbox{infinitely many times}}\Big) + \P\Big( |S_n| \underset{n\to\infty}{\longrightarrow} \infty\Big) = 1
$$
\end{corollary}
\begin{proof}
On the event that $v$ is visited infinitely often but $\pere{v}$ only finitely many times, the urn $U_v$ at vertex $v$ (defined in Proposition \ref{prop:constructurn}) draws color $0$ only finitely many times. If this event has positive probability and since the urns at all the vertices of $\Tfull$ are i.i.d., it follows from Lemma \ref{lemma:draw0} that each one has a positive probability to never draw $0$. But for any $w\in \Tfull$, if $0$ is never drawn from $U_w$, this means that the walk gets trapped inside the subtree rooted at $w$ when it enters it. Since the walk a.s. visits infinitely many vertices (the probability of staying on the same finite number of vertices while accumulating a diverging number of neighbors being zero), this must happen an infinite number of time a.s. Hence, $|S_n| \to \infty$ a.s. However, this is incompatible with the assumption that $v$ is visited infinitely often, hence that event must have null probability. 

To see why the second statement holds true, we recall that $\Tfull$ is countable and observe that on the event $\mathcal{E} \defeq \{ \hbox{there exists a vertex $v\in\Tfull$ that is visited i.o.}\}$, then, by a straightforward induction and thanks to \eqref{eq:vtovo}, the root loop must also be crossed infinitely many times. 
Conversely, on the complementary event $\mathcal{E}^c =\{\hbox{every $v\in\Tfull$ is visited finitely many times}\}$, the number of children of each vertex is necessarily finite, hence the subtree of vertices at distance at most $\ell\geq 1$ from the root is also finite, which implies that $|S_n| \geq \ell$ for $n$ large enough since each vertex of this subtree is visited only finitely many times. This is true for all $\ell$ so we conclude that $|S_n| \to\infty$ on $\mathcal{E}^c$.
\end{proof}

In the rest of the paper, we will be mostly concerned with the case where $\nu$ has finite expectation. In that setting the urn is ``well-behaved'' as the next result shows. 
\begin{lemma}\label{lemma:urnenondegeneree}
Suppose that $\bar\nu < \infty$. Then, the number of colors present in the basic urn increases to infinity and every color is drawn infinitely many times a.s., \emph{i.e.}
$$
\sum_{n\geq 1} \ind{B_n = j} = \infty \quad \hbox{a.s. for all $j\geq 0$.}
$$
\end{lemma}
\begin{proof}
By the law of large numbers, the number of balls in the urn at step $n$ is a.s. asymptotically equivalent to $\bar{\nu} n$ as $n$ tend to infinity. Furthermore, the probability $\P(B_n = 0)$ of picking up a ball of color $0$ at step $n$ is of order $\rho/(\bar{\nu}n)$ so that by the Borel-Cantelli Lemma, a ball of color $0$ is drawn from the urn infinitely many times a.s. Similarly, $\P(B_n = j) \sim 1/(\bar{\nu}n)$ for $j\geq 1$ and $n$ larger than the time at which the color $j$ was introduced in the urn, this said time being finite a.s. Therefore, any color is picked up infinitely many times a.s.
\end{proof}

We now let $\Theta_0 \defeq 0$ and, for $k\geq 1$, we define the stopping time at which color $0$ is drawn from the urn for the $k$-th time:
$$
\Theta_k \defeq \inf\left\{n\geq  1, \sum_{i=1}^n \ind{B_i = 0} = k\right\}.
$$
We also define $N_0 \defeq 0$ and for $k\geq 1$ and $j\geq 1$, 
\begin{align}
\label{def:Nk}N_k & \defeq \substack{\hbox{number of non-zero colors present in the urn at step $\Theta_k$}\\\hbox{(\emph{i.e.} when color $0$  has been drawn $k$ times).}}\\
\label{def:Yki}Y_k^j& \defeq \sum_{i =1}^{\Theta_k} \ind{B_i = j} = 
 \substack{\hbox{number of times color $j$  has been drawn}\\ \hbox{at the $k$-th draw of color $0$.}}
\end{align}
Thanks to the previous lemma, all the quantities above are well-defined and finite a.s. as soon as $\bar\nu < \infty$. Thus, from now on and until section \ref{sec:sec4}, we will always assume that
\begin{equation}\label{eq:nubar_fini}
\bar\nu = \sum_{k} k\nu(k) < \infty.
\end{equation}

\subsection{The branching Markov chain \texorpdfstring{$Z$}{Z}\label{sec:BMC_Z}}

Proposition \ref{prop:constructurn} show that we can construct the TBRW from urn processes but it does not really help to understand the trajectory of the walk $S$ \emph{i.e.} the order in which the urns are chosen. To this end, we now construct a branching Markov chain whose transition kernel $\K$ is defined in terms of the urn process. We set
$$
\begin{cases}
\K(k, [y_1,\ldots, y_j]) = \P(N_k = j \hbox{ and } Y_k^1 = y_1,\ldots, Y_k^j = y_j)&\hbox{for $j\geq 1$,}\\
\K(k, []) = \P(N_k = 0).&
\end{cases}
$$
In words, $\K$ is the kernel which, given a non-negative integer $k$, returns a finite (possibly empty) list representing, for each non-zero color present in the urn, the number of times it has been drawn at the moment when color $0$ is drawn for the $k$-th time. In particular, we have $\K(0, []) = 1$ because no non-zero color is initially present in the urn. Note also that a color $i$ can be present in the urn without having being drawn so that $Y_k^i = 0$. This is not the same thing as the color not yet being present in the urn. 

We use this kernel to construct a tree-indexed particle system, where each particle has a \emph{type} $k \in \N$ and reproduces itself by creating offspring according to $\K$.  More  precisely, we define a branching Markov chain process $Z = (Z(v), v\in \T)$  with the following rules:
\begin{enumerate}
\item[(a)] $\T$ is the rooted planar tree representing the genealogy of the process. A vertex $v\in \T$ is called a \emph{particle} and $Z(v) \in \N$ is its \emph{type}. The \emph{generation} of a particle $v$ corresponds to its distance $|v|$ from the root particle $o$.
\item[(b)] We start initially from a single root particle $o$ with type $Z(o) \in  \N$ (generation $0$). 
\item[(c)] Suppose that we have constructed the process up to generation $n$. Then, all the particles at generation $n$ evolve independently. Each particle $v$ (with $|v|=n$) gives birth to a random finite number of children $\fils{v}{1}, \ldots \fils{v}{j}$ with respective types $Z(\fils{v}{1}),\ldots Z(\fils{v}{j})$. The distribution of the number of children and their types is given by the transition kernel $\K$:
$$
\begin{cases}
\P\big(\hbox{$v$  has $j$ children with } Z(\fils{v}{1}) = z_1,\ldots, Z(\fils{v}{j}) = z_j\;|\; Z(v) = z\big) = \K\big(z, [z_1,\ldots, z_j]\big)&\\
\P\big(\hbox{$v$ does not have any child}\;|\; Z(v) = z\big) = \K\big(z, []\big)&
\end{cases}
$$
\item[(d)] If there is no particle at generation $n$, then we say that the process \emph{dies out} and we stop the construction. 
\end{enumerate}

The process $Z$ can be interpreted as a multi-type branching process with infinitely many types (indexed by $\N$). We observe that, when both $\nu(0) > 0$ and $\nu(1) > 0$, then $Z$ is irreducible in the sense that, for any $z\geq 1$ and any $z_1,\ldots, z_j \geq 0$, a particle of type $z$ has positive probability to have exactly $j$ children with respective types $z_1,\ldots, z_j$. In the general case, additional particles of type $0$ may be unavoidable, but this is not a problem because particles of type $0$ play no role in the genealogy of the process.
\begin{proposition}[irreducibility of $Z$ when omitting particles of type $0$] \label{prop:iredducibilityZ}
Particles of type $0$ in the branching Markov chain $Z$ never reproduce. For any $z\geq 1$ and any $z_1,\ldots, z_j \geq 1$, there exists $n \geq 0$ (which may depend on $z,z_1,\dots,z_j$) such that: 
\begin{equation}\label{eq:irred1}
\P\Big(
\substack{\hbox{$v$  has $j + n$ children with $Z(\fils{v}{1}) = z_1,\ldots, Z(\fils{v}{j}) = z_j$}\\
\hbox{and $Z(\fils{v}{j+1}) = 0 \ldots, Z(\fils{v}{j+n}) = 0$}}
\;\Big|\; Z(v) = z\Big) > 0.
\end{equation}
Furthermore, for all $z\geq 1$, 
\begin{equation}\label{eq:irred2}
\P\Big( \hbox{$v$ has at least one child of type $1$}
\;\Big|\; Z(v) = z\Big) \geq \frac{(1-\nu(0))\rho}{(1+\rho)^2} > 0.
\end{equation}
\end{proposition}
\begin{proof}
The proposition is a direct consequence of the definition of the kernel $\K$ in terms of the urn process. As stated above when defining $\K$, the urn starts without any ball of non-zero color implying $\K(0,[])=1$, which means that a particle of type $0$ has a.s. no progeny. 

Now, given $z \geq 1$, because we assume $\nu(0) < 1$, it is possible to construct a sequence of draws in the urn that creates at least $j$ new colors before drawing a ball of color $0$ for the $z$-th time and such that the first $j$ balls are picked respectively $z_1,\dots,z_j$ times while all balls of color $\geq j+1$ are not picked. This proves \eqref{eq:irred1}. 

Finally, fix $z\geq 1$ and consider the urn process just after a ball of color $0$ has been drawn for the $z-1$-th time. Then, with probability $1-\nu(0)$, at least one new color is added to the urn at that step and such a color is subsequently drawn a geometric number of times, with parameter $\rho/(1+\rho)$, before drawing a ball of color $0$ for the $z$-th time. Therefore, the probability that this ball is drawn exactly once (yielding a child of type $1$) is equal to $\frac{\rho}{(1+\rho)^2}$. This proves \eqref{eq:irred2}.
\end{proof}

Our interest for the branching  Markov chain $Z$ comes from its close relation to the local time of the TBRW stopped at a crossing time of the loop at the root.  Let us define $\tau_k$ as the first time that the tree-builder walk $S_n$ crosses the root loop for the $k$-th time. More precisely, we set $\tau_0 \defeq 0$ and for $k \geq 1$,
\begin{equation}\label{def:tauk}
\tau_k \defeq \inf \left\{n\in \N\,:\, \sum_{i = 1}^n \ind{S_{i-1} = S_{i} = o} = k \right\} 
\end{equation}
with the usual convention that $\inf\emptyset = +\infty$. We also define the local time process $L(\cdot,k)$ of the walk $S$ at time $\tau_k$ on the (oriented) edge of the tree $\T_{\tau_k}$ constructed so far: 
\begin{equation}\label{def:loctimeprocess}
    L(v,k) \defeq \sum_{n=1}^{\tau_k}\ind{S_{n-1} = v \hbox{ and } S_n = \pere{v}}\quad \hbox{for $v\in \T_{\tau_k}$}
\end{equation}
The following ``Ray-Knight type'' result highlights the connection between the local time process $L$ and the branching Markov chain $Z$ defined previously and will enable us to reduce questions about the recurrence/transience of the TBRW to questions pertaining the survival of $Z$.  

\begin{proposition}\label{prop:loctimerec}  Fix $k\geq 1$. The following statements hold true.
\begin{enumerate}
\item[(a)] We have $\P(\tau_k < \infty) = \P(Z \hbox{ dies out } \;|\; Z(o) = k)$.
\item[(b)] The process= $L(\cdot,k)$ under the conditional law $\P(\cdot \;|\; \tau_k < \infty)$ has the same distribution as the process $Z$ under the conditional law $\P(\cdot \;|\;  Z(o) = k \hbox{ and } Z \hbox{ dies out})$.
\item[(c)] The r.v. $\tau_k$ under the conditional law $\P(\cdot \; | \; \tau_k < \infty)$ has the same distribution as the r.v. $k + 2\sum_{v\in \T\setminus\{o\}} Z(v)$ under the conditional law $\P(\cdot \;|\;  Z(o) = k \hbox{ and } Z \hbox{ dies out})$.
\end{enumerate}
\end{proposition}
\begin{remark}
The arguments of the proof provided below follow those given in Section $3.1$ of \cite{BasdevantSingh} to describe the local time process of a multi-excited random walk on a regular tree (again stopped at a crossing of the root loop) \emph{c.f.} in particular Lemma 3.2 of \cite{BasdevantSingh}. Let us point out that the general idea that a non-Markovian walk on an acyclic graph with local interactions can have a local time process with a Markovian structure more amenable to analysis is not new. It traces back at least to the seminal paper of Kesten, Kozlov and Spitzer~\cite{KKS75} to study random walks in random environments. It was subsequently used for different kinds of self-interacting processes, e.g. self-repelling walks \cite{Toth96, TothWerner98} or reinforced walks \cite{BasdevantSingh2}. Although our setting here is slightly different due to the constant evolution of the state space, the general method still applies since the graph is always acyclic. 
\end{remark}
\begin{proof}[Proof of Proposition \ref{prop:loctimerec}]
We fix an integer $d \geq 1$ and we consider a modification $(S^{(d)}, \T^{(d)})$ of the TBRW reflected at height $d$, that is, when the walk is at a site $v$ at distance $|v| = d$ from the root, then it always moves back towards its father $\pere{v}$ at the next step without creating any new site neighboring $v$. Otherwise, if $|v| <d$, it behaves like the regular TBRW. This new walk constructs a tree of height $d$ in which, ultimately, all vertices at height $< d$ have infinite degree (and vertices at height $d$ are leaves). This assertion follows from the same argument use in the proof of Lemma \ref{lemma:urnenondegeneree} to show that the urn process draws each color infinitely many times. Indeed, when the walk is at a given vertex $v$ for the $k$-th time, by the strong law of large numbers, $v$ has $\sim k \bar\nu$ children so that the walk moves to the father $\pere{v}$ with probability $\sim \rho/(\bar{\nu} k)$, and to a given child with probability $\sim 1/(\bar{\nu} k)$. Hence, all those events happen infinitely often a.s. by the Borel-Cantelli Lemma. 

In particular, the walk reflected at height $d$ crosses the root loop infinitely often a.s. so the crossing times $\tau^{(d)}_k$ defined as in \eqref{def:tauk} with $S^{(d)}$ in place of $S$ are well-defined and finite a.s. for every $k\geq 1$. We can similarly define the local time process $L^{(d)}(\cdot, k)$ as in \eqref{def:loctimeprocess} but with $S^{(d)}$ in place of $S$. In words, $L^{(d)}(v,k)$ is the number of crossing for the reflected walk of the oriented edge $v \to  \pere{v}$ before crossing the root loop $k$ times.  Now, the crucial observation is that, since we are on a tree and because the walk starts and ends at the root, then  at time $\tau_k$  every edge (apart from the root loop) has been crossed the same number of times in both directions. Indeed, when the walk crosses an edge $v\to \fils{v}{i}$, then it must necessarily return to $v$ through the edge $\fils{v}{i}\to v$ before returning to the root. Looking at the definition of the TBRW, it is now clear that the dynamics of the walk at any given vertex $v$ is independent of what happens at others vertices (because the walk returns by the same edge it exited) and follows, for vertices at height $< d$, the dynamics of the urn process of Section \ref{subsec:urnprocess}. Putting all these arguments together, we conclude that the process $L^{(d)}(\cdot, k)$ restricted to vertices with height $< d$ and the branching Markov chain $Z$ starting from a particle of type $k$ and considered up to generation $d-1$ have the same law. 

By trivial coupling, the processes $(S_n, \; n\leq \tau_k)$ and $(S^{(d)}_n, \; n\leq \tau^{(d)}_k)$ coincide on the event $\{ \tau_k < \infty \hbox{ and }\max_{n \leq \tau_k} S_n < d \} = \{ \max_{n \leq \tau^{(d)}_k} S^{(d)}_n < d\}$. Therefore, the equality in  distributions stated above shows that
\begin{align*}
\P\big(\tau_k < \infty \hbox{ and } \max_{n \leq \tau_k} S_n < d \big) & = \P\big( \max_{n \leq \tau^{(d)}_k} S^{(d)}_n < d \big) \\ 
&= \P\big(Z \hbox{ dies out before generation $d$} \;|\; Z(o) = k\big)
\end{align*}
which proves statement (a) of the proposition by letting $d$ go to infinity. Statement (b) follows from the same argument since, a.s., the walk and its reflected modification coincide on the event $\{ \tau_k < \infty \}$ for $d$ large enough. Finally, statement (c) is a consequence of the deterministic equality
$$\tau_k = k + 2\sum_{v\in \T_{\tau_k}\setminus\{o\}} L(v,k)$$
which holds true when $\tau_k <\infty$ and summarizes the facts that the walk traverse a single edge at each step and that every edge, apart from the root loop, in traversed the same number of times in both direction up to time $\tau_k$. 
\end{proof}
We can now establish a $0-1$ law for the TBRW. 
\begin{corollary}[0-1 law for the recurrence/transience]\label{cor:01law}
We are in exactly one of the two settings: 
\begin{enumerate}
\item Either $\P(\hbox{$Z$ dies out} \;|\; Z(o) = k) < 1$ for all $k\geq 1$. In that case, the infinite tree $\T_\infty$ created by the tree-builder walk consists of a single infinite line on which are grafted finite trees. We have a.s. that $\lim_\infty |S_n| = +\infty$ (in particular the root loop is visited only finitely many times): we say that the walk is transient.
\item Or $\P(\hbox{$Z$ dies out} \;|\; Z(o) = k) = 1$ for all $k\geq 1$. In that case, the tree $\T_\infty$ created by the walk is the full infinite tree a.s. (\emph{i.e.} each vertex has an infinite number of children) and the walk $(S_n)$ visits every vertex of $\T_\infty$ (and the root loop, by Corollary~\ref{cor:rec01}) infinitely many times a.s: we say that the walk is \emph{recurrent}.
\end{enumerate}
\end{corollary}

\begin{proof}
Let us first assume that  $\P(\hbox{$Z$ dies out} \;|\; Z(o) = k) = 1$ for all $k\geq 1$. Then, by Proposition \ref{prop:loctimerec}, we have $\tau_k < \infty$ a.s. for all $k \geq 1$ hence the walk returns to the root infinitely often a.s. But, because we know that the TBRW urn ultimately contains infinitely many colors and that each one is drawn infinitely many times a.s. (Lemma \ref{lemma:urnenondegeneree}), this implies, by induction on the height, that every vertex of the tree is visited infinitely many times and has infinitely many children which are themselves visited infinitely often a.s. 

Conversely, let us assume that $\P(\hbox{$Z$ dies out} \;|\; Z(o) = k_0) < 1$ for some $k_0\geq 1$. Thanks to the irreducibility property \eqref{eq:irred1} of $Z$ of Proposition \ref{prop:iredducibilityZ}, this assumption must, in fact, hold true for all $k \geq 1$. Furthermore, the extinction probability of $Z$ is non-increasing in $k$ in view of (a) of Proposition \ref{prop:loctimerec} and the fact that $\tau_{k+1} > \tau_k$. Therefore, on the one hand, we have for all $k\geq 1$:
$$
\P(\hbox{$Z$ dies out} \;|\; Z(o) = k) \leq \P(\hbox{$Z$ dies out} \;|\; Z(o) = 1) =: \alpha < 1. 
$$
On the other hand, recalling again that the TBRW urn creates infinitely many colors and that each one is drawn infinitely many times, we find that, for any fixed $N$, 
$$
\lim_{k\to\infty} \P(\hbox{$o$ has at least $N$ children with non-zero type} \;| \;Z(o) =k)  = 1 
$$
Therefore, using the branching property of $Z$, we have
$$
\P(\hbox{$Z$ dies out} \;|\; Z(o) = k) \leq \P(\hbox{$o$ has less than $N$ children with non-zero type} \;| \;Z(o) =k)  + \alpha^N 
$$
from which we deduce that (taking $N$ arbitrary large and then $k\rightarrow \infty$)
$$
\P(\tau_k < \infty) = \P(\hbox{$Z$ dies out} \;|\; Z(o) = k) \underset{k\to\infty}{\longrightarrow} 0.
$$
This means that the walk returns to the root only finitely many times a.s. By the same argument as before, it follows by induction that every vertex is visited a finite number of times (and has a finite number of children). Finally, since the walk $(S_n)$ performs only nearest-neighbor moves, we must have $\lim_n |S_n| = \infty$ a.s. and the tree $\T_\infty$ must be one-ended. 
\end{proof}

\section{Proof of theorem \ref{maintheo1}\label{sec:sec3}}

Thanks to Proposition \ref{prop:loctimerec}, we can relate the transience of the TBRW to the survival of the branching Markov chain $Z$. In order to study whether $Z$ dies out or not, we will make use of the following criterion for characterizing the survival of the multi-type branching process $Z$ in terms of the spectral radius of its expectation matrix.
\begin{proposition}\label{prop:lyapounov2}
Consider the urn process of section \ref{subsec:urnprocess} and recall the definitions of $N_k$ and $Y_k^i$ given in \eqref{def:Nk} and \eqref{def:Yki}. Let $f: \N^* \to \R_+$ be a non-negative and non-null function. 
\begin{enumerate}
\item Suppose that there exists $\lambda > 1$ such that 
\begin{equation}\label{eq:lyaponov_supercritical}
\sum_{i=1}^\infty  f(i)\E\left[\sum_{\ell=1}^{N_i} \ind{Y_i^\ell = k}\right] \geq \lambda f(k) \quad\hbox{ for all $k \geq 1$.}
\end{equation}
Then, the branching Markov chain $Z$ is super-critical: it has positive probability to survive indefinitely starting from any non-empty initial configuration. 
\item Suppose that $\sum f(k) < \infty$ and that there exists $\lambda < 1$ such that 
$$
\sum_{i=1}^\infty  f(i)\E\left[\sum_{\ell=1}^{N_i} \ind{Y_i^\ell = k}\right] \leq \lambda f(k) \quad\hbox{ for all $k \geq 1$.}
$$
Then, the branching Markov chain $Z$ is sub-critical: starting from any finite configuration of particles, it dies out almost surely and the total population over the entire lifetime of the process has finite expectation. 
\item  Suppose that $\sum f(k) < \infty$ and 
$$
\sum_{i=1}^\infty  f(i)\E\left[\sum_{\ell=1}^{N_i} \ind{Y_i^\ell = k}\right] = f(k) \quad\hbox{ for all $k \geq 1$.}
$$
Then, the branching Markov chain $Z$ is critical: starting from any finite configuration of particles, it dies out almost surely but the total population over the entire lifetime of the process has infinite expectation. 
\end{enumerate}
\end{proposition}

\begin{proof}
We first prove $1.$ Let $M = (M_{i,j})_{i,j \geq 1}$ denote the expectation matrix for the branching Markov chain $Z$ (ignoring particles of type $0$ which do not reproduce): 
$$
M_{i,j} \defeq \hbox{expected number of children of type j for a particle of type i} = \E\left[\sum_{\ell=1}^{N_i} \ind{Y_i^\ell = j}\right]
$$
Let us first suppose that $M_{i_0,j_0} = \infty$ for some $i_0,j_0 \geq 1$. Then, by irreducibility of $Z$ (Proposition \ref{prop:iredducibilityZ}) we have $M^2_{i_0,i_0} \geq M_{i_0,j_0} M_{j_0,i_0} = \infty$ so the expected number of particles of type $i_0$ created by a single particle of the same type after two generations is infinite. By comparison with a classical (single-type) branching process, this shows that $Z$ is supercritical. 

We now assume $M_{i,j} < \infty$ for all $i,j$.  Inequality \eqref{eq:lyaponov_supercritical} can be rewritten in matrix form $f M \geq \lambda f$ so it shows that the spectral radius of $M$ is smaller or equal to $1/\lambda$. For $L\geq 1$, let $M_{|L}$ denote the finite matrix obtained by restricting $M$ to $\llbracket  1,L  \rrbracket$. Due to~\eqref{eq:irred1}, the matrix $M_{|L}$ is irreducible (all its coefficients are strictly positive), therefore, according to~\cite[Theorem 6.8]{Seneta}, as $L$ goes to infinity, the spectral radius of $M_{|L}$ decreases to that of $M$. Since $\lambda>1$, we can choose $L_0$ large enough so that the spectral radius of $M_{|L_0}$ is $1/\lambda_0 < 1$. Consider now the modified branching Markov chain $Z_{|L_0}$ constructed from $Z$ by instantaneously removing all particles of types outside $\llbracket  1,L_0  \rrbracket$. It is a classical, irreducible, multi-type branching process (with a finite number of types) with expectation matrix $M_{|L_0}$ which has its Perron-Frobenius eigenvalue $\lambda_0 > 1$. Standard results about multi-type branching processes (see for instance~\cite[Theorem II.7.1]{Harris}) ensure that $Z_{|L_0}$ is super-critical and therefore $Z$ also has a positive probability of survival. 

We now establish $2.$ Since $\sum f(k) < \infty$, we can assume without loss of generality that $\sum f(k) = 1$ (because $f$ is assumed to be non-zero). Starting the branching Markov chain $Z$ from a single particle with initial random type distributed as $f$, we expect at time $n$, for each $k\geq 1$, no more than $\lambda^n f(k)$ particles of type $k$ on average. Therefore, the expected total number of non-zero type particles that will ever exist in the system is $\sum_n \sum_k f(k)\lambda^n = 1/(1-\lambda) < \infty$. In particular, its expectation is finite so that the total number of particles is a fortiori an a.s. finite random variable. Hence the process dies out a.s. By irreducibility, this result remains true when starting from any finite configuration of particles. 

Finally, in case $3.$, 
the same argument as above shows that starting from a single particle with type distributed as $f$, the expected number of non-zero type particles at time $n$ is always $1$. Now, assume that for some $k\geq 1$ such that $f(k)>0$, there is a positive probability $p_k$ that $Z$ started from a single particle of type $k$ survives indefinitely. On this event, there are a.s. infinitely many generations containing at least one vertex of type 1, by~\eqref{eq:irred2} and the Borel-Cantelli Lemma. Since $Z$ is irreducible (hence a vertex of type 1 can have an arbitrarily large offspring of vertices of type $k$), by the same argument, for every $\ell\geq 1$ there is a.s. at least one generation containing $\ell$ vertices of type $k$, each having then probability $p_k$ to survive indefinitely. Therefore, starting from a single vertex of type $k$ (resp. of type distributed as $f$), there is a probability $p_k>0$ (resp.~$f(k)p_k>0$) that the sizes of the successive generations of $Z$ grow to infinity. But it is standard to check that this contradicts the fact that the expected size of each generation is 1. 
Hence we must have $p_k=0$, and thus $p_j=0$ for all $j\geq 1$ by irreducibility: the process $Z$ dies a.s., regardless of the type of the initial particle. 
\end{proof}

Recall the definition of the urn process of section \ref{subsec:urnprocess}.

\begin{lemma} \label{lemma:EN}For any $\rho,\nu$ such that $\rho>\bar\nu$, we have:
\begin{equation}\label{eq:EN}
\E[N_k - N_{k-1}]= \bar \nu \left ( \frac{\rho}{\rho-\overline{\nu}}\right )^k 
\end{equation}
In particular, we have $\E[N_k] < \infty$ for all $k$. 
\end{lemma} 
\begin{proof}
We consider the following variant of the urn process defined in Section \ref{subsec:urnprocess}. First, we start with $\rho$ balls of color $0$ and a random number $r$ of balls with colors $1, \ldots , r$, where $r$ is distributed as $\nu$. Second, whenever we pick a ball of color $0$, we do not add any additional balls in the urn. The rest of the dynamics is unchanged. Let $\bar N_k$ be the number of balls in this modified urn when we pick the ball $0$ for the $k$-th time. We claim that 
\begin{equation}\label{eq:ENbar}
\E[\bar N_k]= \bar \nu\left (\frac{\rho}{\rho-\bar \nu}\right )^k<\infty,
\end{equation}
which we will show at the end of the proof.

We compare this modified urn process with the original urn process from Section~\ref{subsec:urnprocess}. Consider the following change of color attribution for the original urn: instead of naming the colors in $\N$ we name them from $\{0\}\cup (\N^* \times \N^*)$ such that:
\begin{itemize}
\item[(a)] when we add $\ell$ new balls after picking the ball of color $(i,j)$, we give them the colors $(i,k), \ldots , (i,k+\ell-1)$, where $k$ is the smallest positive integer such that $(i,k)$ is not yet in the urn;
\item[(b)] when we add $\ell$ new balls after picking color $0$ for the $i$-th time, we give them the colors $(i,1), \ldots , (i,\ell)$.
\end{itemize}

Now, we claim that for every $k\geq 1$, 
\begin{equation}\label{eq:ENtotal}
\E[N_k]=\sum_{j=1}^k\E[\bar N_j]    
\end{equation}
which thus implies (by subtracting the LHS for two consecutive values of $k$) that for every $k\geq 1$,
\begin{equation}\label{eq:ENdiff}
    \E[ N_k- N_{k-1}]=\E[\bar N_k].
\end{equation}
Together with~\eqref{eq:ENbar}, this concludes the proof of the lemma. 

\medskip

\noindent\textbf{Proof of~\eqref{eq:ENbar}.} In order to inject some independence into the modified urn process, instead of drawing the balls one by one, we put an exponential clock of rate $1$ for each ball of color $i\geq 1$ and an exponential clock of rate $\rho$ for color $0$ and pick them when their clock rings. Recall that we start at time $0$ with $\rho$ balls of color $0$ and $r$ ball of respective colors $1,2,\ldots,r$ with $r$ random and distributed as $\nu$. When color $0$ rings, we do nothing. When a color $i\geq 1$ rings, we add a random number $r'$ of balls, where $r'$ is again distributed as $\nu$ and we give new colors to each of these balls. 

We focus on $C_t$, the number of colors $i\geq 1$ at time $t$. In a small interval $[t,t+dt]$, with probability $C_tdt$ one of the clocks corresponding to a color $i\geq 1$ rings and the expected number of new colors then added is $\bar \nu$. Hence,
\[ \frac{d}{dt}\E[C_t]=\bar\nu\E[C_t]. \]
Since $\E[C_0]=\bar \nu, 
$ we find $\E[C_t]=\bar \nu e^{\bar\nu t}$.

Next we look at the stopping time $T_k$ when $0$ rings for the $k$-th time. By definition, $T_k$ is a sum of $k$ independent exponential random variables of parameter $\rho$. Hence, $T_k$ has probability density $\frac{\rho}{\Gamma(k)} (\rho x)^{k-1}e^{-\rho x},\, x\geq 0$. Therefore, by independence of the clocks,
\[\E[\bar N_k]=\E[C_{T_k}]= \bar \nu\frac{\rho}{\Gamma(k)}\int_0^\infty e^{\bar\nu x} (\rho x)^{k-1}e^{-\rho x} dx = \bar \nu\left (\frac{\rho}{\rho-\bar \nu}\right )^k.\]

\medskip

\noindent\textbf{Proof of~\eqref{eq:ENtotal}.}
For $k\geq j\geq 1$, let $N_{k,j}$ be the number of balls in the original urn process at time $\Theta_k$ (i.e.~when we pick 0 for the $k$-th time) whose color is in $j\times \mathbb{N}^*$. Remark that it is enough to show the following statement: for all $k\geq j\geq 1$, $N_{k,j}$ is distributed as $\bar N_{k-j}$. 
\\
For fixed $k\geq j\geq 1$, when we pick color 0 for the $j$-th time, we then add the first $r$ balls of color in $j\times \mathbb{N}$, where $r\sim \nu$. Since the balls with color $\N^*\backslash \{j\}\times \N^*$ do not affect the probability that the next color picked among $\{0\}\cup \{j\times \mathbb{N}^*\}$ is 0, we have that $N_{k,j}$ is distributed as $\bar N_{k-j}$. 
\end{proof}

\begin{lemma}\label{lemma:geom}
For any $(\rho,\nu)$ such that $\rho>\bar\nu$, any non-negative function $f$ and any $k\geq 0$, we have the identity
$$
\E\left[\sum_{i=1}^{N_k} f(Y_k^i)\right] \; = \; \bar\nu\sum_{j=1}^k \E\big[f(\zeta_1+\ldots+\zeta_{j})\big] \Big(\frac{\rho}{\rho-\bar\nu}\Big)^{k+1-j}
$$
where $(\zeta_i)_{i\geq 1}$ is a sequence of i.i.d. random variables with geometric distribution with parameter $\frac{\rho}{1+\rho}$ starting from $0$ \emph{i.e.} $\P(\zeta_i = j) = \frac{\rho}{1+\rho}\big(\frac{1}{1+\rho}\big)^j$ for $j\geq 0$.
\end{lemma}
\begin{proof}
We have
$$
\E\left[\sum_{i=1}^{N_k} f(Y_k^i)\right] = \E\left[\sum_{i=1}^{\infty} f(Y_k^i)\ind{i \leq N_k}\right]
= 
\sum_{i=1}^\infty  \sum_{\ell=1}^k \E\left[ f(Y_k^i)\ind{N_{\ell-1} < i \leq N_\ell}\right]
$$
Now, for any fixed $i\geq 1$ and $1\leq \ell \leq k$, conditionally on the event $\{N_{\ell-1} < i \leq N_\ell\}$,  the random variable $Y_k^i$ has the same law as $\zeta_1+\ldots+\zeta_{k-\ell+1}$. To see why this is true, notice that, in the urn process, on $\{N_{\ell-1} < i \leq N_\ell\}$, color $i$ appears in the urn between the $\ell-1$ and the $\ell$ draw of color $0$ and therefore $Y_k^i$ is the number of times the ball of color $i$ will be drawn before drawing color $0$ another $k-\ell+1$ times. Now, irrespectively of other balls being added to the urn, the probability of drawing a ball of color $0$ is always $\rho$ times larger than that of drawing a ball of color $i$. Therefore, $Y_k^i$ follows the law of the number of failures before having $k-\ell+1$ successes in a sequence of independent Bernoulli trials with parameter $\rho/(1+\rho)$. This is the negative binomial distribution which has the same law as $\zeta_1+\ldots+\zeta_{k-\ell+1}$.  Therefore, 
\begin{align*}
\E\left[\sum_{i=1}^{N_k} f(Y_k^i)\right] &= \sum_{i=1}^\infty  \sum_{\ell=1}^k \E\left[ f(\zeta_1+\ldots+\zeta_{k-\ell+1})\right]\P(N_{\ell-1} < i \leq N_\ell)\\
& = \sum_{\ell=1}^k \E\left[ f(\zeta_1+\ldots+\zeta_{k-\ell+1})\right]\sum_{i=1}^\infty  \P(N_{\ell-1} < i \leq N_\ell)\\
& = \sum_{\ell=1}^k \E\left[ f(\zeta_1+\ldots+\zeta_{k-\ell+1})\right] \E[N_{\ell}- N_{\ell-1}]\\
& = \sum_{\ell=1}^k \E\left[ f(\zeta_1+\ldots+\zeta_{k-\ell+1})\right] \bar\nu\Big(\frac{\rho}{\rho- \bar\nu}\Big)^{\ell} \\
& = \bar\nu\sum_{j=1}^k \E\left[ f(\zeta_1+\ldots+\zeta_{j})\right] \Big(\frac{\rho}{\rho- \bar\nu}\Big)^{k-j+1}
\end{align*}
where we used \eqref{eq:EN} to compute $\E[N_{\ell}-N_{\ell-1}]$. 
\end{proof}

\begin{proposition}
For any $(\rho,\nu)$ such that $\rho>\bar\nu$ and $0 < s < \frac{\rho - \bar\nu}{\rho}$, we have, for all $k\geq 1$, 
$$
\sum_{n=1}^\infty  s^n \E\left[\sum_{\ell=1}^{N_n} \ind{Y_n^\ell = k}\right] = \frac{\bar\nu \rho^2 s}{(\rho- \bar\nu-\rho s)(\rho+1-\rho s)} \left(\frac{1}{\rho+1-\rho s}\right)^k
$$
\end{proposition}

\begin{proof}
We simply compute, using Lemma~\ref{lemma:geom} with $f(x)=\mathbf{1}_{x=k}$ in the first line:
\begin{align*}
\sum_{n=1}^\infty  s^n \E\left[\sum_{\ell=1}^{N_n} \ind{Y_n^\ell = k}\right] & =
\bar\nu\sum_{n=1}^\infty  s^n \sum_{j=1}^n \E\left[ \ind{\zeta_1 + \ldots + \zeta_j = k}\right]  \Big(\frac{\rho}{\rho-\bar\nu}\Big)^{n-j+1}\\
&= \bar\nu\sum_{j=1}^\infty \P(\zeta_1 + \ldots + \zeta_j = k) \Big(\frac{\rho}{\rho-\bar\nu}\Big)^{-j+1}\sum_{n\geq j}\Big(\frac{s\rho}{\rho-\bar\nu}\Big)^{n}\\
&=  \frac{\bar\nu \rho}{\rho-\bar\nu-\rho s} \sum_{j=1}^\infty \P(\zeta_1 + \ldots + \zeta_j = k) s^j\\
&= \frac{\bar\nu\rho}{\rho-\bar\nu-\rho s} \sum_{j=1}^\infty \Big(\frac{\rho}{\rho+1}\Big)^j\Big(\frac{1}{\rho+1}\Big)^{k} \binom{k+j-1}{j-1}s^j\\
&= \frac{\bar\nu\rho^2s}{(\rho-\bar\nu-\rho s)(\rho+1)^{k+1}} \sum_{i=0}^\infty \Big(\frac{\rho s}{\rho+1}\Big)^{i} \binom{k+i}{i}\\
&= \frac{\bar\nu\rho^2s}{(\rho-\bar\nu-\rho s)(\rho+1)^{k+1}} \left(\frac{1}{1-\frac{\rho s}{\rho+1}}\right)^{k+1}\\
& =  \frac{\bar\nu \rho^2 s}{(\rho-\bar\nu-\rho s)(\rho+1-\rho s)} \left(\frac{1}{\rho+1-\rho s}\right)^k.
\end{align*}

\end{proof}

\begin{corollary}\label{cor:vecpropre}Let $(\rho,\nu)$ be such that $\rho > 1+\bar\nu$, and define $f(k) \defeq \frac{1}{\rho^k}$. We have, for all $k\geq 1$, 
$$
\sum_{i=1}^\infty  f(i) \E\left[\sum_{\ell=1}^{N_i} \ind{Y_i^\ell = k}\right] = \frac{\bar\nu}{\rho-\bar\nu-1} f(k).
$$
\end{corollary}

\bigskip

\begin{proof}[Proof of Theorem \ref{maintheo1}]
Let us first assume that $\rho > 1+\bar\nu$. We observe that
$$
 \frac{\bar\nu}{\rho-\bar\nu-1}
 \begin{cases}
 < 1 &\hbox{if $\rho > 2\bar\nu + 1$,}\\
 = 1 &\hbox{if $\rho = 2\bar\nu + 1$,}\\
 > 1 &\hbox{if $\rho < 2\bar\nu + 1$.}
 \end{cases}
$$
Combining Corollary \ref{cor:vecpropre}, Proposition \ref{prop:lyapounov2} and Proposition \ref{prop:loctimerec}, we conclude that the TBRW is recurrent for $\rho \geq 1 + 2\bar\nu$ and transient for $ \rho \in (1+\bar\nu,  1+ 2\bar\nu)$. Furthermore, for $\rho > 1 + 2\bar\nu$, Proposition \ref{prop:lyapounov2} states that the expected total population of the branching Markov chain $Z$ is finite, which in view of (c) of Proposition \ref{prop:loctimerec} shows that $\E[\tau_k] < \infty$ for all $k$ so the TBRW is positive recurrent. Conversely, when $\rho = 1 + 2\bar\nu$, the total population of $Z$ has infinite expectation hence $\E[\tau_k] = \infty$ for all $k$ so the TBRW is null recurrent in that case. Finally, when $\rho \leq 1+\bar\nu$, then the walk is transient thanks to Proposition \ref{prop:coupling_nurho} by comparison with a $(\rho', \nu)$-TBRW with $\rho' \in (1+\bar\nu, 1+2\bar\nu)$. 

The proof of Theorem \eqref{maintheo1} will be complete once we lift assumption \eqref{eq:nubar_fini} that $\nu$ has finite expectation. Thus, we now assume that $\bar\nu = \infty$ and show that the TBRW is transient: by truncation, we can find another distribution $\nu'$ such that $\nu' \prec \nu$ that also satisfies $\rho < 1 + 2\bar\nu' < \infty$ hence we know that the  $(\rho, \nu')$-TBRW is transient. Combining Proposition \ref{prop:coupling_nurho} and Corollary \ref{cor:rec01}, we conclude that the  $(\rho, \nu)$-TBRW is also transient.
\end{proof}

\section{Asymptotic behavior of the TBRW\label{sec:sec4}}
 We prove Theorem \ref{theo:LLN_TCL} in section \ref{sec:theo_lln_tcl} and Theorem \ref{theo:rec_log} in section \ref{subsec:theo_rec_log}.
\subsection{Cut times for the TBRW\label{subsec:cutimes}}

The proof of the law of large numbers and the central limit theorem stated in Theorem \ref{theo:LLN_TCL} is based on the classical path decomposition of the walk into excursions between cut times (also called regeneration times). The method is rather generic and was popularized to prove ballistic behavior for random walks in random environments (see for instance \cite{Piau,Sznitman00, SznitmanZerner99,  Zeitouni02}) and then subsequently applied successfully to other models such as the excited random walk \cite{BasdevantSingh, BerardRamirez07} and, more recently, for the original model of the TBRW (without bias) \cite{ribeiroTBRW}. We quickly recall here the definition of the cut times and how the finiteness of the first/second moment relates to the existence of a LLN/CLT for the TBRW.

Recall that $(S_n, \T_n)$ denote the $(\rho, \nu)$-TBRW and that $|S_n|$ is the height of the walk at time $n$ (\emph{i.e.} its distance from the root). The first cut time $C_1$ is defined as
$$
C_1 \defeq \inf\left\{n > 0\, : \, S_n \notin \{S_0,\ldots, S_{n-1} \}\hbox{ and } |S_k| \geq |S_n| \hbox{ for all $k\geq n$}\right\}.
$$
In words, $C_1$ (if it exists) is the first time that the walk visits a new vertex $v$ and then never leaves the subtree rooted at $v$. We point out that $C_1$ is not a stopping time because its definition depends on the future of the walk, but it can be decomposed into a stopping time (the walk visits a new vertex $v$) and an event that does not depend on the past of the walk (the walk never goes back to the father of $v$ afterward). Moreover, when $S$ is transient, the latter event has positive probability which does not depend on $v$ and equals 
\begin{equation}\label{def:alpha}
\alpha \defeq \P(\tau_1 = \infty) \in (0,1).
\end{equation}
Our interest in cut-times comes from the following standard result. 
\begin{proposition}\label{prop:cuttimes}
    Suppose that the first cut time $C_1$ for the TBRW $S$ is well-defined and has finite expectation $\E[C_1] < \infty$. Then, there exists $v > 0$ such 
    $$
    \frac{S_n}{n} \overset{\hbox{\tiny{a.s.}}}{\underset{n\to\infty}{\longrightarrow}} v.
    $$
    Furthermore, if $\E[C_1^2] < \infty$, then there exists $\sigma > 0$ such that
    $$
    \frac{|S_n| - nv}{\sqrt{n}} \overset{\hbox{\tiny{law}}}{\underset{n\to\infty}{\longrightarrow}} \mathcal{N}(0,\sigma^2). 
    $$
\end{proposition}
\begin{proof}
The proof is classical. Under the assumption that $C_1$ exists and because the event $\{ \tau_1 = \infty\}$ has positive probability $\alpha$, there exists a.s. an infinite number of cut times $C_1 < C_2 < \ldots$. Furthermore, by definition of the cut times, the sequences $(C_{i+1} - C_{i})_{i\geq 1}$  and $(|S_{C_{i+1}}| - |S_{C_{i}}|)_{i\geq 1}$ are i.i.d. and have the same distribution, respectively, as  $C_1$ and $S_{C_1}$ conditioned on the event $\{ \tau_1 = \infty\}$. Then, the LLN/CLT for the walk now follows from the LLN/CLT applied to classical i.i.d. sequences of random variables that have a first/second moment together with a change of time $n \leftrightarrow C_n \sim n\E[C_1]$. We leave out the details and we refer the reader to \cite{Sznitman00} and \cite{SznitmanZerner99} for additional explanations.  
\end{proof}

\subsection{Proof of Theorem \ref{theo:LLN_TCL}\label{sec:theo_lln_tcl}}

In all this section, we assume that the $(\rho, \nu)$-TBRW is transient \emph{i.e.} 
$$\rho < 1 + 2\bar\nu.$$ 
We also do not require that $\bar\nu < \infty$ anymore so the proof given here will hold in full generality.  In view of Proposition \ref{prop:cuttimes}, The theorem will be established once we show that the first cut time $C_1$ is well-defined and has a finite second moment. We will prove the stronger result:
\begin{proposition}\label{prop:integreC1}
The first cut time $C_1$ is finite a.s. and, for all $n$ large enough, we have
$$
\P(C_1 > n) \leq e^{-n^{1/6 + o(1)}}. 
$$
In particular, $C_1$ has moments of all orders.
\end{proposition}

The proposition relies on two technical estimates whose proofs are specific to the model of the TBRW (Lemma \ref{lem:tau1condi} and Lemma \ref{lem:badrnbis}) while the main argument to control the tail of $C_1$ using a decomposition of the path into ``failed excursions" until the first cut time is standard. We first show that the walk must visit new vertices often. Let 
$$R_n \defeq \#\{S_i,\, 0\leq i \leq n\}$$
denote the range (\emph{i.e.} number of distinct vertices visited) of $S$ up to time $n$. We point out that this quantity is usually strictly smaller than the number of vertices of $\T_n$.

\begin{lemma}\label{lem:badrnbis}
There exists $\delta>0$ such that for all $n\geq 1$:
\[
\P(R_n\leq \delta n^{1/2})\leq e^{-n^{1/2}+o(1)}.
\]
\end{lemma}
\begin{proof}
Let $\delta >0$, and let $v_i$ be the $i$-th vertex visited by $S$. Note that $\{R_n\leq \delta n^{1/2}\}\subseteq \{\exists j\leq n, v_j\text{ is visited at least $\delta^{-1}n^{1/2}$ times by $(S_k)_{1\leq k\leq n}$}\}$. The number $k$ of neighbors created during the first $\lfloor \delta^{-1}n^{1/2}/2\rfloor$ visits to any given vertex stochastically dominates a binomial  $\emph{Bin}(\lfloor \delta^{-1}n^{1/2}/2\rfloor,1-\nu(0))$ variable. Hence, choosing $\delta <(1-\nu(0))^2/4$, by Azuma's inequality (or standard large deviation bounds) and a union bound on $j\leq n$,  
\[
\P(\exists j\leq n, \text{$v_j$ is visited $\lfloor \delta^{-1}n^{1/2}/2\rfloor$ times, creating $\leq \delta^{-1/2}n^{1/2}$ neighbors})\leq e^{-n^{1/2}+o(1)}. 
\]
But if a vertex has already at least $ \delta^{-1/2}n^{1/2}$ neighbors, for each subsequent visit of $S$, the probability to exit via a not yet visited vertex, if the range of the walk so far is $\leq \delta n^{1/2}$, is therefore at least $1-\delta n^{1/2} /(\delta^{-1/2}n^{1/2}) \geq 1-\delta^{3/2} \geq 3/4$. Thus, letting $Y\sim \emph{Bin}(\lfloor \delta^{-1} n^{1/2}/2\rfloor , 3/4)$, 
\begin{align*}
&\P(R_n\leq \delta n^{1/2}) \leq \P\Big(\substack{\hbox{$\exists j\leq n$, $v_j$ is visited $\delta^{-1/2} n^{1/2}$ times,}\\\hbox{$S$ leaving $v_j$ via $\leq \delta n^{1/2}$ distinct neighbors}} \Big)
\\
&\leq e^{-n^{1/2}+o(1)}+n\P(Y\leq \delta n^{1/2})\leq e^{-n^{1/2}+o(1)}, 
\end{align*}
where the last inequality follows again, for $\delta$ small enough, from large deviation bounds on binomial variables.  
\end{proof}

We also need a control for the range of the walk at time $\tau_1$ when it first crosses the root loop, conditionally on this return time being finite. 

\begin{lemma}\label{lem:tau1condi} 
For all $n$ large enough, we have
$$
\P(R_{\tau_1} > n , \tau_1 < \infty) \leq e^{-n^{1/2 + o(1)}}.
$$
\end{lemma}

\begin{proof}
We bound the range $R_{\tau_1}$ by bounding separately the depth and  the width of the subtree of vertices visited by the walk up to time $\tau_1$. We start by looking at the depth of this subtree: fix $k \geq 1$ and consider the stopping time $h_k \defeq \inf(n \geq 0 : |S_n| = k)$. Define the event $\mathcal{E}_k \defeq \{h_k < \tau_1 < \infty\}$. On $\mathcal{E}_k$, the walk reaches height $k$ at vertex $v_k \defeq S_{h_k}$. Let $o = v_0 , v_1, \ldots , v_k$ be the vertices along the ancestor line from the root $o$ to $v_k$. After time $h_k$ and because $\tau_1$ is finite on $\mathcal{E}_k$, the walk must, at some time, jump from $v_k$ to $v_{k-1}$, then later on it must jump from $v_{k-1}$ to $v_{k-2}$ and so on.

Now, for each $i$, when the walk goes back to $v_i$ for the first time, it has probability $1-\nu(0)$ to create at least one new leaf. Furthermore, the probability that a given leaf is visited before crossing the edge from $v_i$ to $v_{i-1}$ is exactly $\frac{1}{1+\rho}$. Finally, when the leaf is  visited for the first time, because the walk is transient, it has probability $\alpha \defeq \P(\tau_1 = \infty) > 0$ to never return to $v_i$. Putting all this together, we conclude that the probability that the walk starting from $v_k$ crosses back the root loop at some point is bounded above by the product of the probability of crossing each edge $(v_i , v_{i-1})$ successively, and therefore
\begin{equation*}
\P(\mathcal{E}_k) \leq \Big(1 - \frac{\alpha (1-\nu(0))}{1+\rho}\Big)^{k}.
\end{equation*}
We now look at the width of the excursion. Fix $\ell \geq 1$ and let $G_\ell \defeq \sharp \{ S_k : k \leq \tau_1 \hbox{ and } |S_k| = \ell\}$ denote the number of vertices visited by the walk until time $\tau_1$ which are at distance $\ell$ from the root. Fix $m \geq 1$ and consider $\mathcal{G}_{m, \ell} \defeq \{ G_\ell \geq m \hbox{ and } \tau_1 < \infty\}$. On this event, the walks visits $m$ distinct sites at height $\ell$ but, at each first visit of a site $v$ at height $\ell$, it has probability $\alpha$ never to go back to $\pere{v}$ - which is necessary to visit another site of the same height. Therefore, we find that
\begin{equation*}
\P(\mathcal{G}_{m,\ell}) \leq \alpha^{m}
\end{equation*}
Finally, we decompose
\begin{align*}
\P(R_{\tau_1} > n , \tau_1 < \infty) &\leq \P\big( h_k < \tau_1 < \infty) + \P(\exists \ell < k, G_\ell \geq n/k, \tau_1 < \infty)\\
&\leq \P\big( \mathcal{E}_k) + \sum_{\ell = 1}^k\P(\mathcal{G}_{\floor{n/k},\ell})\\
&\leq \Big(1 - \frac{\alpha (1-\nu(0))}{1+\rho}\Big)^{k} + k \alpha^{\floor{n/k}},
\end{align*}
and the result follows choosing $k=\lfloor\sqrt{n}\rfloor$ 
\end{proof}

\begin{proof}[Proof of Proposition \ref{prop:integreC1} (hence of Theorem \ref{theo:LLN_TCL})]
Now that we have obtained the critical estimates on $R_n$ and $R_{\tau_1}$, the proof of the proposition is \emph{mutatis mutandis} the same as that of Proposition 6.1 of \cite{BasdevantSingh} with Lemma 6.3 and 6.4 of \cite{BasdevantSingh} replaced respectively by Lemma \ref{lem:tau1condi} and \ref{lem:badrnbis}. We provide below the main arguments for the sake of completeness.  

Let $(\mathcal{F}_n)_{n\geq 0}$ denote the natural filtration for $S$. We define by induction two interlaced sequences of stopping times: $A_0=B_0=0$  and for $i\geq 1$,
$$
A_i:=\inf\{n >B_{i-1}: S_n\notin \{S_0, S_1, \ldots, S_{n-1} \} \quad ; \quad B_i:=\inf\{n>A_i, |S_n| < |S_{A_i}|\}.
$$
In other words $A_i$ is the first time after $B_{i-1}$ when the walk visits a new vertex, and $B_i$ is the first time after $A_i$ that it goes back to the father of $S_{A_i}$. Note that we have $B_i=\infty$ with positive probability since $(S_i)_{i\geq 0}$ is a.s. transient.

We know that the walk visits infinitely many sites a.s., so that we have the equality of events $\{ B_{k-1} < \infty\} = \{A_{k} < \infty\}$ a.s. Furthermore, because the walk evolves on a new subtree during each excursion $[A_k, B_k)$, it follows that, conditionally on $\{A_k < \infty\}$, the random variable $B_k - A_k$ is independent of $\mathcal{F}_{A_k}$ and distributed as $\tau_1$. In particular, we have $\P(B_{k} = \infty \,|\, A_k < \infty, \mathcal{F}_{A_k}) = \P(\tau_1 = \infty)$. Therefore, by induction, it follows that the r.v. $\theta  \defeq\inf\{i : B_i = \infty\}$ has geometric distribution with parameter $\alpha \defeq \P(\tau_1 = \infty) > 0$. This proves that $C_1 = A_{\theta}$ is finite a.s.

Now, we consider $R_{C_1}$ the number of site visited by the walk at time $C_1$. Observing that $S$ only visits new sites during the excursions $[A_i, B_i)$, we obtain the identity
$$
R_{C_1} = 1 + \sum_{1\leq i <\theta} (V_i+1)
$$
where $V_i$ is the number of distinct sites visited during the excursion $[A_i, B_i)$. Now, using again that those excursions occur on distinct subtrees, we deduce that, conditionally on $\theta$, the sequence $(V_1,\ldots , V_{\theta-1})$ is i.i.d. with $V_i$ having the same distribution as the r.v. $R_{\tau_1}$ conditioned on $\{\tau_1 < \infty\}$. Therefore, we get by a union bound:
\begin{align*}
\P\big(V_1 + \ldots + V_{\theta - 1} > n\big) & \leq  \P\big(\theta > \lfloor n^{1/3} \rfloor\big) + \P\big(V_1 + \ldots + V_{\theta - 1} > n, \; \theta \leq \lfloor n^{1/3} \rfloor \big) \\
 & \leq  \P\big(\theta > \lfloor n^{1/3} \rfloor\big) +  \lfloor n^{1/3} \rfloor \P\Bigg(R_{\tau_1} > \frac{n}{\lfloor n^{1/3} \rfloor} \, \Bigg|\, \tau_1 < \infty\Bigg)\\
& \leq (1-\alpha)^{\lfloor n^{1/3} \rfloor} + \frac{\lfloor n^{1/3}\rfloor}{1-\alpha} e^{- \Big(\frac{n}{\lfloor n^{1/3} \rfloor}\Big)^{1/2 + o(1)}}\\
& \leq e^{-n^{1/3 + o(1)}}
\end{align*}
where we used the Lemma \ref{lem:tau1condi} and the fact that $\theta$ has geometric distribution with parameter $\alpha$ to upper bound the probabilities on the second line. Finally, using Lemma \ref{lem:badrnbis}, we conclude that
\begin{align*}
\P\big(C_1 > n\big) & \leq \P\big(C_1 > n,\; R_n > \delta n^{1/2} \big) + \P(R_n \leq \delta n^{1/2} )\\
&\leq \P(R_{C_1} > \delta n^{1/2} ) + \P(R_n \leq \delta n^{1/2} )\\
&\leq e^{-\floor{\delta n^{1/2}-2}^{1/2 + o(1)}} + e^{-n^{1/2 + o(1)}}\\
&\leq e^{-n^{1/6 + o(1)}}.
\end{align*}
\end{proof}

\subsection{Proof of Theorem \ref{theo:rec_log}\label{subsec:theo_rec_log}}

We assume in this last section that the $(\rho,\nu)$-TBRW is positive recurrent, \emph{i.e.} 
$$1+2\bar\nu < \rho.$$
We will use the following lemma which shows that the return times of the walk through the root loop increase exponentially fast. 
\begin{lemma}\label{lemma:exptau}
    There exists $\delta= \delta(\rho,\nu) > 0$ such that
    $$
    \tau_k \geq e^{\delta k}\quad\hbox{for all $k$ large enough a.s.}
    $$
\end{lemma}

\begin{proof}
We crudely lower bound $\tau_k$ by the number of jumps from the root $o$ to a child of $o$ before jumping through the root loop for the $k$-th time. Define the filtration $\mathcal H_k \defeq \sigma((S_i,\T_i)_{i\leq \tau_k})$ and let $d_{k} \defeq \deg_o(\T_{\tau_k})$ denote the number of children of the root $o$ at time $\tau_k$. We also define $\Delta_k:= \sum_{i=\tau_k}^{\tau_{k+1}-1}\ind{S_i = o}$ the number of visits to the root vertex between two consecutive passages through the root loop. We observe that, conditionally on $\mathcal H_k$, the r.v. $\Delta_k$  stochastically dominates a geometric random variable with parameter $\rho/(d_k+\rho)$.  Moreover, every time we are at $o$ between $\tau_k$ and $\tau_{k+1}$, we have a positive probability $\beta \defeq 1 - \nu(0) >0$ of creating at least $1$ child. Thus, taking $d_k$ large, we find that 
\begin{align*} \P(d_{k+1}-d_k\leq \beta d_k \;|\;\mathcal H_k) & \leq 
\P\left (\Delta_k<2 d_k  \;|\;\mathcal H_k \right )+
\P \left (d_{k+1}-d_k\leq \beta d_k ,\; \Delta_k\geq 2  d_k \;|\; \mathcal H_k \right )\\
& \leq \P\Big(\hbox{Geom}\Big(\frac{\rho}{\rho+d_k}\Big) < 2 d_k\;\Big|\; \mathcal H_k \Big) + \P\Big( \hbox{Binom}(2 d_k, \beta) \leq \beta d_k\;\Big|\; \mathcal H_k\Big)\\
&\underset{d_k\to\infty}{\longrightarrow} 1 - e^{-2\rho} < 1.  
\end{align*}
Therefore, for any $p_0 < e^{-2\rho}$, by comparison with a sequence of i.i.d. variables with $\text{Ber}(p_0)$ distribution, the strong law of large numbers implies that 
$$
\liminf_n \frac{1}{n}\sum_{k\leq n} \ind{d_{k+1}-d_k > \beta d_k}  \geq p_0 \quad\hbox{a.s.}
$$
This in turn implies the existence of $\delta > 0$ such that a.s., $d_k \geq e^{\delta k}$ for all $k$ large enough. Finally, denoting $(\xi_n)_{ n\geq 0}$ the i.i.d. sequence with distribution $\nu$ where $\xi_n$ is the number of children added to the root at its $n$-th visit, we observe that
\begin{multline*}
\P\big( \tau_k \leq e^{\delta k / 3}  , d_k \geq e^{\delta k} \Big) 
\leq \P\Big(\xi_1 + \ldots \xi_{\floor{e^{\delta k / 3}}} \geq e^{\delta k}\Big) \\
\leq \floor{e^{\delta k / 3}} \P\Big(\xi_i \geq \frac{e^{\delta k}}{\floor{e^{\delta k / 3}}}\Big)
\leq \frac{\bar \nu \floor{e^{\delta k / 3}}^2}{e^{\delta k}}  =  O(e^{-\delta k/3}),
\end{multline*}
using Markov's inequality. 
This proves that the events $\{\tau_k \leq e^{\delta k / 3}  , d_k \geq e^{\delta k}\}$ occur only finitely many times a.s. Hence, we conclude that we also have $\tau_k \geq e^{\delta k /3}$ for all $k$ large enough a.s. 
\end{proof}

\begin{proof}[Proof of Theorem \ref{theo:rec_log}]
We first establish the upper bound. We assume here that we are in the positive recurrent setting $\rho > 1 + 2\bar\nu$. Fix $k\geq 1$ and consider $f(i) \defeq \frac{\rho^k}{\rho^i}$. Corollary \ref{cor:vecpropre} states that $f$ is a left eigenvector associated with the eigenvalue $\lambda \defeq \frac{\bar\nu}{\rho-\bar\nu - 1} \in (0,1)$ for the expectation matrix of the multi-type branching process $Z$. This means that if we start $Z$ from an initial distribution of particles with, in average, $f(i)$ particles of type $i$ for all $i\geq 1$, then at generation $\ell$, there are, in average, exactly $\lambda^\ell f(i)$ particles of type $i$. But since $f(k) = 1$, this means by positivity that starting $Z$ from a single particle of type $k$, the average number of particle of type $i \geq 1$ at generation $n$ is at most 
 $\lambda^n f(i)$. Therefore, we find that
\begin{multline*}
\P(|\T_{\tau_k}| > \ell) \leq
\P\Big(\substack{\hbox{there exists a particle of} \\ \hbox{non-zero type at generation $\ell$}} \,\Big| \, Z(o) = k\Big)\\
\leq \E\Big[
\substack{\hbox{number of non-zero type} \\ \hbox{particles at generation $\ell$}}
\,\Big| \, Z(o) = k\Big] \leq \sum_{i\geq 1} f(i)\lambda^\ell \leq \frac{\rho^k \lambda^\ell}{\rho-1}.
\end{multline*}
Let us now fix $c > \frac{\log(\rho)}{\log(1/\lambda)}$. It follows that $\P(|\T_{\tau_k}| > c k)$ decreases exponentially fast, hence $|\T_{\tau_k}| \leq c k$ for all $k$ large enough a.s. On the other hand, Lemma \ref{lemma:exptau} states that $\tau_k \geq e^{\delta k}$ for all $k$ large enough a.s. Combining these two facts, we conclude that we can choose $c_2$ large enough such that $|\T_{n}| \leq c_2 \log n$ for all $n$ large enough a.s. This proves the upper bound. 

In order to establish the lower bound, we compare $(|S_n|, n\geq 0)$ with a biased random walk. We observe that at each step, the probability that the walk increases its height at the next step is always larger than the probability that the walk creates at least one new child (with probability $1-\nu(0)$) and then move to a child rather than the father at the next step (with probability at least $1/(1+\rho)$ since the vertex now has at least one child).  Therefore, 
$$
\P\big(|S_{n+1}| - |S_{n}| = 1 \,\big|\, |S_0|,|S_1|,\ldots, |S_n|\big) \geq q \defeq \frac{1-\nu(0)}{1+\rho} .
$$
This show that we can couple $(|S_n|, n\geq 0)$ with a $q$-biased nearest-neighbor random walk $(X_n,n\geq 0)$ reflected at $0$ in such way that $|S_n| \geq X_n$ for all $n$. We conclude using classical estimates on random walks that we can choose $c_1 > 0$ small enough such that
$$
|\T_{n}| \geq \sup_{i\leq n} |S_i| \geq \sup_{i\leq n} |X_i| \geq c_1 \log n\quad\hbox{for all $n$ large enough a.s.}
$$
\end{proof}

\section*{Acknowledgement}

This work is a collaborative project initiated during an open problem session at the conference \emph{Processes on Graphs and Maps days} (also fondly known as \emph{La Bardade au CIRM}), which took place in November 2024 and was funded by ANR 19-CE40-0025 ``ProGraM''. The authors would like to thank once again Eleanor Archer and Laurent Ménard for organizing this great meeting. 

We also warmly thank Alessandra Caraceni for stimulating discussions and for making simulations of the process during the early stage of this work.

\bibliographystyle{plain}
\bibliography{biblio}
\end{document}